\documentclass[a4paper,12pt]{article}

\usepackage{graphicx}
\usepackage{graphics} % for pdf, bitmapped graphics files
\usepackage{epsfig} % for postscript graphics files
\usepackage{times} % assumes new font selection scheme installed
\usepackage{amsmath} % assumes amsmath package installed
\usepackage{amsfonts}  % assumes amsmath package installed
\usepackage{amsthm}

\newcommand{\Pp}{{\text{Poincar\'e disk}}}

\newcommand{\RR}{{\mathbb R}}
\newcommand{\NN}{{\mathbb N}}

\newcommand{\abs}[1]{\lvert#1\rvert}
\newcommand{\norm}[1]{\lVert#1\rVert}

\newcommand{\dotex}{{\frac{d}{dt}}}

\newcommand{\tr}[1]{\text{Tr}\left(#1\right)}

\newcommand{\FixedRank}{S_+(r,n)}

\newcommand{\OG}[1]{{\mathcal{O}({#1})}}

\newcommand{\Sym}{{\mathrm{Sym}}}

\newcommand{\set}[2]{{\{{#1}\ \mathrm{s.t.}\ {#2}\}}}

\newtheorem{thm}{Theorem}
\newtheorem{prop}{Proposition}

\newtheorem{lem}{Lemma}

    \date{}  % include this line if your document contains figures
\begin{document}

\title{Stochastic gradient descent on Riemannian manifolds}
\author{S. Bonnabel
\thanks{Robotics lab, Math\'ematiques et syst\`emes,  Mines ParisTech, 75272 Paris CEDEX, France (e-mail: silvere.bonnabel@ mines-paristech.fr).}}

\maketitle

\begin{abstract}Stochastic gradient descent is a simple approach to find the local minima of a cost function whose evaluations are corrupted by noise. In this paper, we develop a procedure extending stochastic gradient descent algorithms to the case where the function is defined  on  a Riemannian manifold. We prove that, as in the Euclidian case, the gradient descent algorithm converges to a critical point of the cost function.  The algorithm has numerous potential applications, and is   illustrated here by four examples. In particular a novel gossip algorithm on the set of covariance matrices is derived and tested numerically.
\end{abstract}

\section{Introduction}

Stochastic approximation provides a simple approach, of great
practical importance,  to find the local minima of a function whose
evaluations are corrupted by noise. It has had a long history in
optimization  and control with numerous applications (e.g. \cite{ljung,benveniste,tsitsiklis}). To demonstrate the
main ideas on a toy example,  we briefly mention  a traditional
procedure to optimize the ballistic trajectory of a projectile in a
fluctuating wind. Successive gradient corrections (i.e. corrections
proportional to the distance between the projectile impact and the target)
are performed on the angle at which the projectile is launched. With
a decreasing step size tending to zero, one can reasonably hope the
launching angle will  converge to a fixed value which is such that the corresponding impacts are centered on the target {on average}.  One of the first
formal algorithm of this kind is the Robbins-Monro algorithm \cite{robbins}, which
dates back to the 1950s. It proves that for a smooth cost function
$C(w)$ having a unique minimum, the algorithm $w_{t+1}=w_t-\gamma_t
h_t(w_t)$, where $h_t(w_t)$ is a noisy evaluation of the gradient of
$C$ at $w_t$, converges in quadratic mean to the minimum, under
specific conditions on the sequence $\gamma_t$.

Although stochastic gradient has found applications  in control, system identification,  and filtering theories  (for instance  a Kalman filter for  noisy observations of a constant process is a Robbins-Monro algorithm),  new challenging applications  stem from the active machine learning community. The work of L. Bottou, a decade  ago \cite{bottou-98}, has popularized the stochastic gradient approach, both to address the online learning problem (identification of a constant parameter in real time from noisy output measurements) and large-scale learning (with ever-increasing data-sets, approximating the cost function with a simpler appropriate stochastic function can lead to a reduced numerical complexity). Some recent  problems have been strong drivers for the development of new estimation methods, such as the one proposed in the present paper, dealing with stochastic gradient descent on manifolds.

The paper is organized as follows. In Section \ref{proof:sec} general stochastic gradient descent algorithms on Riemannian manifolds are introduced. The algorithms already used in  \cite{meyer-11,meyer-icml,amari-98,arnaudon-10,oja-92} can all be cast in the proposed general framework. The main algorithms are completely intrinsic, i.e. they do not depend on a specific embedding of the manifold or on a choice of local coordinates.

In section \ref{conv:sec} the convergence properties of the algorithms are analyzed. In the Euclidian case,  almost sure (a.s.) convergence of the parameter to a critical point of the gradient of the cost function is well-established  under reasonable assumptions (see e.g. \cite{bottou-98}), but this result has never been proven to hold for non-Euclidian spaces. In this paper, almost sure convergence of the proposed algorithms is obtained under several assumptions, extending the results of the Euclidian case to the Riemannian case.

In Section \ref{examples:sec} the algorithms and the convergence results of the preceding sections are applied to four examples. The first example revisits the celebrated Oja  algorithm \cite{oja-92} for online principal component analysis (PCA). This algorithm can be cast in our versatile framework, and its convergence properties immediately follow from the theorems of Section \ref{conv:sec}. Moreover, the other results of the present paper allow to define alternative algorithms for online PCA with guaranteed convergence properties. The second example is concerned with the randomized computation of intrinsic means on a hyperbolic space, the $\Pp$. This is a rather tutorial example, meant to illustrate the assumptions and results of the third theorem of Section \ref{conv:sec}. The convergence follows from this theorem.   The last two examples are more detailed and include numerical experiments. The third example is concerned with a particular algorithm of  \cite{meyer-11}. The goal is to identify a positive semi-definite matrix (a kernel or a Mahalanobis distance) from noisy measurements. The theoretical convergence results of Section \ref{conv:sec}  allow to complete the work of  \cite{meyer-11}, and simulations illustrate the convergence properties. The last example is concerned with a consensus application on the set of covariance matrices (see e.g. \cite{ljung2}).  A novel randomized gossip algorithm based on the Fisher Information Metric is proposed. The algorithm has a meaningful statistical interpretation, and admits several invariance  and guaranteed convergence properties that follow from the results of Section \ref{conv:sec}. As the state space is convex, the usual gossip algorithm \cite{boyd-gossip} is well defined and can be implemented on this space. Simulations indicate the proposed Riemannian consensus algorithm converges faster than the usual gossip algorithm.

Appendix A briefly presents some links with information geometry and Amari's natural gradient. Appendix B contains a brief recap of differential geometry.  Preliminary results can be found in  \cite{bonnabel-et-al:mtns,gretsi}.

\section{Stochastic gradient on Riemannian manifolds}\label{proof:sec}

\subsection{Standard stochastic gradient in $\RR^n$}
Let $C(w)=\mathbb E_z Q(z,w)=\int Q(z,w) dP(z)$ be a three times continuously
differentiable cost function, where $w\in\RR^n$ is a minimization parameter, and  $dP$ is a probability measure on a measurable space $\mathcal Z$. Consider the optimization problem
\begin{align}\label{pb1:eq}
&\min_w C(w)\end{align}
  In stochastic approximation, the cost function cannot be computed explicitly as the distribution $dP$ is assumed to be unknown. Instead, one has access to a sequence of independent observations $z_1,z_2\cdots$ of a random variable  drawn with probability law $dP$.  At each time step $t$, the user can compute the so-called  loss function $Q(z_t,w)$ for any parameter $w\in\RR^n$. The loss can be viewed as an approximation of the (average) cost function $C(w)$ evaluated under the input $z_t\in\mathcal Z$. Stochastic gradient descent is a standard technique to treat this problem. At each step the algorithm receives an input $z_t$ drawn according to $dP$, and performs a gradient descent on the approximated cost  $
w_{t+1}=w_t-\gamma_t H(z_t,w_t)
$
where   $H(z,w)$ can be viewed as the gradient of the loss, i.e.,   on average $\mathbb E_z H(z,w)=\int H(z,w) dP(z)=\nabla C(w)$. As   $C$ is not convex in many applications, one can not hope for a much better result than almost sure (a.s.) convergence of    $C(w_t)$ to some value $C_\infty$, and convergence of  $\nabla C(w_t)$ to $0$.  Such a result holds under a set of standard assumptions, summarized in e.g. \cite{bottou-98}. Note that, a.s. convergence is a very desirable property for instance in online estimation, as  it ensures asymptotic convergence  is always achieved in practice.

\subsection{Limits of the approach: a motivating example}\label{motivating:sec}

A topical problem that has attracted a lot of attention in the
machine learning community over the last years is low-rank matrix
estimation (or matrix completion, which can be viewed as the matrix counterpart of sparse approximation problems) and in particular the collaborative filtering problem:
given a matrix $W_{ij}^*$ containing the preference ratings of users
about items (movies, books), the goal is to compute personalized
recommendations of these items. Only a small subset of entries
$(i,j)\in \Xi$ is known, and there are many ways to complete the
matrix. A standard approach to overcome this ambiguity, and
to filter the noise, is to constrain the state space by assuming the
tastes of the users are explained by a reduced number of criteria
(say, $r$). This yields the following non-linear optimization
problem
$$
\min_{W\in\RR^{d_1\times d_2}} \quad \sum_{(i,j)\in\Xi}(W_{ij}^*-W_{ij})^2\quad s.t. \quad \text{rank}(W)=r
$$
The matrix being potentially of high dimension ($d_1\simeq 10^5,d_2\simeq 10^6$ in the so-called Netflix prize problem), a standard method to reduce the computational burden is to draw random elements of  $\Xi$, and  perform gradient descent ignoring the remaining entries. Unfortunately the updated matrix
 $
W-\gamma_t\nabla_W (W_{ij}^*-W_{ij})^2
$
does not have rank $r$. Seeking the matrix of  rank $r$ which best approximates it   can be numerically costly, especially for very large $d_1,d_2$. A more natural way to enforce the rank constraint is to endow the parameter space with a Riemannian metric, and to perform a gradient step within the manifold of fixed-rank matrices.  In \cite{meyer-icml} this approach has led to  stochastic gradient algorithms that compete with state of the art methods. Yet a convergence proof is still lacking. The convergence results below are general, and in Section \ref{PSD:sec} they will be shown to apply  to this problem for the particular case of $W^*$  being symmetric positive definite.

\subsection{Proposed general stochastic gradient algorithm on Riemannian manifolds}

In this paper we propose a new procedure to address problem \eqref{pb1:eq} where $C(w)=\mathbb E_z Q(z,w)$ is a three times continuously
differentiable cost function and where $w$ is now a minimization parameter belonging to a smooth  connected  Riemannian manifold  $\mathcal M$. On  $\mathcal M$, we propose to replace the usual update with the following  update
\begin{align}\label{manifold:grad:eq}
w_{t+1}=\exp_{w_t}(-{\gamma_t} H(z_t,w_t))
\end{align}
where $\exp_w$ is the exponential map at $w$, and $H(z,w)$  can be viewed as the Riemannian gradient of the loss,  i.e., we have on average $\mathbb E_z H(z,w)=\int H(z,w) dP(z)=\nabla C(w)$ where $\nabla C(w)$ denotes the Riemannian gradient of $C$ at $w\in\mathcal M$. The proposed update \eqref{manifold:grad:eq} is a straightforward transposition of the standard gradient update in the Euclidian case. Indeed, $H(z,w)$ is a tangent vector to the manifold that describes the direction of steepest descent for the loss. In update  \eqref{manifold:grad:eq}, the parameter moves along the geodesic emanating from the current parameter position $w_t$, in the direction defined by $H(z_t,w_t)$ and with intensity $\norm{H(z_t,w_t}$. If the manifold at hand is $\RR^n$ equipped with the usual Euclidian scalar product, the geodesics are straight lines, and the definitions coincide. Note that, the procedure here is  totally intrinsic, i.e. the algorithm is completely independent of the choice of local coordinates on the manifold.

In many cases, the exponential map is not easy to compute (a calculus of variations problem must be solved, or the Christoffel symbols need be known), and it is much easier and much faster to use a first-order approximation of the exponential, called a retraction. Indeed a retraction $
R_w(v):T_w {\mathcal M}\mapsto \mathcal M
$ maps the tangent space at $w$ to the manifold, and it is  such that $d(R_w(tv),\exp_w(tv))=O(t^2)$. It yields the alternative update
\begin{align}\label{manifold:retraction:eq}
w_{t+1}=R_{w_t}(-{\gamma_t} H(z_t,w_t))
\end{align}
Let us give a simple example to illustrate the ideas: if the manifold were the sphere $\mathbb S^{n-1}$ endowed with the natural metric inherited through immersion in $\RR^n$, a retraction would consist of a simple addition in the ambient space $\RR^n$ followed by a projection onto the sphere. This is a numerically very simple operation that avoids calculating the geodesic distance explicitly. See the Appendix for more details on Riemannian manifolds.

\section{Convergence results}\label{conv:sec}

In this section, the convergence of the proposed algorithms \eqref{manifold:grad:eq} and \eqref{manifold:retraction:eq} are analyzed. The parameter is proved to converge almost surely to a critical point of the cost function in various cases and under various conditions. More specifically,  three general results are derived.

In Subsection \ref{compact:sec}) a first general result is derived: when the parameter $w\in\mathcal M$ is proved to remain in a compact set, the algorithm \eqref{manifold:grad:eq}  converges a.s. under standard conditions on the step size sequence. This theorem applies in particular to all connected compact manifolds. Important examples of such manifolds in applications are the orthogonal group, the group of rotations, the sphere, the real projective space, the Grassmann and the Stiefel manifold.  In Subsection \ref{retraction:sec}),  the result is proved to hold when a twice continuously differentiable retraction is used instead of the exponential map.

Finally, in Subsection \ref{hadamard:sec}), we consider a slightly modified version of algorithm \eqref{manifold:grad:eq} on specific non positively curved Riemannian manifolds.    The  step size $\gamma_t$ is  adapted at each step in order to take into account the effects of negative curvature that tend to destabilize the algorithm. Under a set of mild assumptions naturally extending those of the Euclidian case, the parameter is proved to a.s. remain in a compact set, and thus a.s. convergence is proved. Important examples of such manifolds are the  Poincar\'e disk or the  Poincar\'e half plane, and the space of real symmetric positive definite matrices $P_+(n)$. The sequence of step sizes $(\gamma_t)_{t\geq 0}$ will satisfy the usual condition in stochastic approximation:
 \begin{align}\label{step:eq}
 \sum\gamma_t^2<\infty\quad\text{and}\text\quad\sum\gamma_t=+\infty
 \end{align}

\subsection{Convergence on compact sets}\label{compact:sec}

The following theorem proves the a.s. convergence of the algorithm under some assumptions when the trajectories have been proved to remain in a predefined compact set at all times.  This is of course the case if  $\mathcal M$ is compact.

\begin{thm}\label{thm1}Consider the algorithm \eqref{manifold:grad:eq} on a connected Riemannian manifold $\mathcal M$ with injectivity radius uniformly bounded from below by $I>0$. Assume the sequence of step sizes $(\gamma_t)_{t\geq 0}$ satisfy the standard condition \eqref{step:eq}. Suppose there exists a compact set $K$ such that $w_t\in K$ for all $t\geq 0$. We  also suppose that the gradient is bounded on $K$, i.e. there exists $A>0$ such that for all $w\in K$ and $z\in\mathcal Z$ we have $\norm{H(z,w)}\leq A$. Then  $C(w_t)$ converges a.s. and $\nabla C(w_t)\to 0$ a.s.
\end{thm}

\begin{proof}
The proof builds upon the usual proof in the Euclidian case (see  e.g. \cite{bottou-98}). As the parameter is proved to remain in a compact set, all continuous functions of the parameter can be bounded. Moreover, as $\gamma_t\to 0$, there exists $t_0$ such that for $t\geq t_0$ we have $\gamma_t A<I$. Suppose now that $t\geq t_0$, then there exists a geodesic $\exp(-s\gamma_t{H(z_t,w_t)})_{0\leq s\leq 1}$ linking $w_t$ and $w_{t+1}$ as $d(w_t,w_{t+1})<I$. $C(\exp(-\gamma_t{H(z_t,w_t)}))=C(w_{t+1})$ and thus the Taylor formula implies that (see Appendix)
 \begin{equation}\begin{aligned}\label{taylor2:eq}
C(w_{t+1})-C(w_t)&\leq -\gamma_t\langle H(z_t,w_t),\nabla C(w_t)\rangle
 \\& \quad+ \gamma_t^2\norm{H(z_t,w_t)}^2k_1
\end{aligned}\end{equation}where $k_1$ is an upper bound on the Riemannian Hessian of $C$ in the compact set $K$.
Let  $\mathcal F_t$ be the increasing sequence of  $\sigma$-algebras generated by the  variables available just before time $t$:
 $$
\mathcal F_t=\{z_0,\cdots,z_{t-1}\}
$$$w_t$ being computed from $z_0,\cdots,z_{t-1}$,  is measurable $\mathcal F_t$.   As $z_t$ is independent from $\mathcal F_t$  we have
$ \mathbb E[\langle H(z_t,w_t),\nabla C(w_t)\rangle|\mathcal F_t] =\mathbb E_z[\langle H(z,w_t),\nabla C(w_t)\rangle] = \norm{\nabla C(w_t)}^2$.
 Thus \begin{align}\label{inee:eq}
\mathbb E(C(w_{t+1})-C(w_t)|\mathcal F_t)\leq -{\gamma_t} \norm{\nabla C(w_t)}^2+{\gamma_t}^2A^2k_1
\end{align}as $\norm{H(z_t,w_t)}\leq A$. As $C(w_t)\geq 0$, this proves $C(w_{t})+\sum_t^\infty\gamma_k^2A^2k_1$ is a nonnegative supermartingale, hence it converges a.s. implying that $C(w_t)$ converges a.s. Moreover  summing the inequalities we have
\begin{equation}\begin{aligned}
\sum_{t\geq t_0}{\gamma_t} \norm{\nabla C(w_t)}^2&\leq -\sum_{t\geq t_0}\mathbb E(C(w_{t+1})-C(w_t)|\mathcal F_t)\\&\quad+\sum_{t\geq t_0} {\gamma_t}^2A^2k_1
\end{aligned}\label{rev:ineq}\end{equation}
Here we would like to prove the right term is bounded so that the left term converges. But the fact that $C(w_t)$ converges does not imply it has bounded variations. However, as in the Euclidian case, we can use a theorem by D.L. Fisk \cite{fisk} ensuring that $C(w_t)$ is a quasi martingale, i.e., it can be decomposed into a sum of a martingale and a process whose trajectories are of bounded variation. For a random variable $X$, let $X^+$ denote the quantity $\max (X,0)$.\begin{prop}\label{fisk:prop}\text{[Fisk (1965)]}
Let $(X_n)_{n\in\NN}$ be a non-negative stochastic process with bounded positive variations, i.e.,  such that $\sum_{0}^\infty
\mathbb E([\mathbb E(X_{n+1}-X_n)|\mathcal F_n)]^+)<\infty$. Then the process is a quasi-martingale, i.e.
$$
\sum_{0}^\infty|\mathbb E[X_{n+1}-X_n|\mathcal F_n]|<\infty \quad\text{a.s. , and $X_n$ converges a.s.}
$$\end{prop}
Summing \eqref{inee:eq} over $t$, it is clear that $C(w_t)$ satisfies the proposition's assumptions, and thus $C(w_t)$ is a quasi-martingale, implying $\sum_{t\geq t_0}{\gamma_t} \norm{\nabla C(w_t)}^2$ converges a.s. because of inequality \eqref{rev:ineq} where the central term can be bounded by its absolute value which is convergent thanks to the proposition.
But, as $\gamma_t\to 0$, this does not prove $\norm{\nabla C(w_t)}$ converges a.s. However,  if $\norm{\nabla C(w_t)}$ is proved to converge a.s., it can only converge to 0 a.s. because of condition \eqref{step:eq}.

Now consider the nonnegative process
 $
p_t=\norm{\nabla C(w_t)}^2
$. Bounding the second derivative of $\norm{\nabla C}^2$ by $k_2$, along the geodesic linking $w_t$ and $w_{t+1}$, a Taylor expansion yields
$
p_{t+1}-p_t\leq -2\gamma_t \langle\nabla C (w_t),(\nabla_{w_t}^2 C)H(z_t,w_t)\rangle+({\gamma_t})^2 \norm{H(z_t,w_t)}^2k_2
$, and thus  bounding from below the Hessian of $C$ on the compact set by $-k_3$ we have $
\mathbb E(p_{t+1}-p_t|\mathcal F_t)\leq 2\gamma_t \norm{\nabla C(w_t)}^2 k_3+{\gamma_t}^2 A^2k_2
$. We just proved the sum of the right term is finite.  It implies $p_t$ is a quasi-martingale,  thus it implies a.s. convergence of $p_t$ towards a value $p_\infty$ which can only be 0.

\end{proof}

  \subsection{Convergence with a retraction}\label{retraction:sec}

  In this section, we prove Theorem \ref{thm1} still holds when a retraction is used instead of the exponential map.

\begin{thm}\label{thm3}
 Let $\mathcal M$ be  a connected Riemannian manifold
 with injectivity radius uniformly bounded from below
 by $I>0$. Let $R_w$ be a twice continuously differentiable
 retraction, and consider the update \eqref{manifold:retraction:eq}.
 Assume the sequence of step sizes $(\gamma_t)_{t\geq 0}$ satisfy the
  standard condition \eqref{step:eq}. Suppose there exists a compact set
  $K$ such that $w_t\in K$ for all $t\geq 0$. We suppose also that the
  gradient is bounded in $K$, i.e. for $w\in K$ we have
   $\forall z~\norm{H(z,w)}\leq A$ for some $A>0$. Then
    $C(w_t)$ converges a.s. and $\nabla C(w_t)\to 0$ a.s.\end{thm}

\begin{proof}
 Let $w_{t+1}^{exp}=\exp_{w_t}(-{\gamma_t} H(z_t,w_t))$. The proof essentially relies on the fact that the points $w_{t+1}$ and $w_{t+1}^{exp}$, are close to each other on the manifold for sufficiently large $t$. Indeed,  as the retraction is twice continuously differentiable there exists $r>0$ such that $d(R_w(sv),\exp_w(sv))\leq rs^2$ for $s$ sufficiently small, $\norm v=1$, and $w\in K$. As for $t$ sufficiently large $\gamma_t A$ can be made arbitrarily small (in particular smaller than the injectivity radius), this implies  $d(w_{t+1}^{exp},w_{t+1})\leq {\gamma_t}^2rA^2$.

We can now reiterate the proof of Theorem \ref{thm1}.   We have $C(w_{t+1})-C(w_t)\leq |C(w_{t+1})-C(w_{t+1}^{exp}))|+C(w_{t+1}^{exp})-C(w_t)$.  The term $C(w_{t+1}^{exp})-C(w_t)$ can be bounded as in \eqref{taylor2:eq} whereas we have just proved  $|C(w_{t+1})-C(w_{t+1}^{exp}))|$ is bounded by $ k_1 r{\gamma_t}^2A^2 $  where $k_1$ is a bound on the Riemannian gradient of $C$ in $K$. Thus $C(w_t)$ is a quasi-martingale and $
\sum_1^\infty \gamma_t \norm{\nabla C(w_t)}^2<\infty
$.  It means that if $\norm{\nabla C(w_t)}$ converges, it can only converge to zero.

Let us consider the variations of the function $p(w)=\norm{\nabla C(w)}^2$. Writing $p(w_{t+1})-p(w_t)\leq |p(w_{t+1})-p(w_{t+1}^{exp}))|+p(w_{t+1}^{exp})-p(w_t)$ and bounding the first term of the right term  by $k_3r\gamma_t^2A^2$ where $k_3$ is a bound on the gradient of $p$, we see the inequalities of Theorem \ref{thm1} are unchanged  up to second order terms in $\gamma_t$. Thus $p(w_t)$  is a quasi-martingale and thus converges. \end{proof}

\subsection{Convergence on Hadamard manifolds}\label{hadamard:sec}

In the previous section, we proved convergence as long as the
parameter is known to  remain in a compact set. For some manifolds,
the algorithm can be proved to converge without this assumption.
This is the case for instance in the Euclidian space, where the
trajectories can be proved to be confined to a compact set under a
set of conditions \cite{bottou-98}. In this section, we extend those
conditions to the important class of Hadamard  manifolds, and we
prove convergence. Hadamard manifolds are complete, simply-connected
Riemannian manifolds with nonpositive sectional curvature. In order
to account for curvature effects, the  step size must be slightly
adapted at each iteration. This step adaptation yields a more
flexible algorithm, and allows to relax one of the standard
conditions even in the Euclidian case.

Hadamard manifolds have strong properties. In particular, the exponential map at any point is globally invertible (e.g. \cite{Oneill-book}). Let $D(w_1,w_2)=d^2(w_1,w_2)$ be the squared geodesic distance. Consider the following assumptions, which can be viewed as an extension of the usual ones in the Euclidian case:
\begin{enumerate}
\item There is a point $v\in\mathcal M$ and $S>0$ such that  the opposite of the gradient points towards $v$ when  $d(w,v)$ becomes   larger than $\sqrt S$ i.e. $$\inf_{D(w,v)>S}\langle \exp_{w}^{-1}(v),\nabla C(w)\rangle<0$$
\item There exists a lower bound on the sectional curvature denoted by $\kappa<0$.
\item There exists a continuous function $f:\mathcal M\mapsto \mathbb R$ that satisfies \begin{align*}&f(w)^2~\geq \max\{1, \\&\mathbb E_z\bigl(\norm{H(z,w)}^2(1+\sqrt{\abs{\kappa}}(\sqrt{ D (w,v)}+ \norm{H(z,w)}))\bigr),\\&\mathbb E_z\bigl((2\norm{H(z,w)}\sqrt{ D (w,v)}+\norm{H(z,w)}^2)^2\bigr)\}\end{align*}
\end{enumerate}

\begin{thm}\label{thm2}
Let $\mathcal M$ be a Hadamard manifold. Consider the optimization problem \eqref{pb1:eq}. Under assumptions 1-3, the modified algorithm
\begin{align}\label{manifoldf:grad:eq}
w_{t+1}=\exp_{w_t}(-\frac{\gamma_t}{f(w_t)} H(z_t,w_t))
\end{align} is such that $C(w_t)$ converges a.s. and $\nabla C(w_t)\to 0$ a.s.
\end{thm}

Assumptions 1-3 are  mild assumptions that encompass the Euclidian
case. In this latter case Assumption 3 is usually replaced with the
stronger condition  $E_z(\norm{H(z,w)}^k)\leq A+B\norm{w}^k$ for
$k=2,3,4$ (note that, this condition immediately implies the existence
of the function $f$). Indeed, on the one hand our general procedure
based on the adaptive step ${\gamma_t}/{f(w_t)}$ allows to relax
this standard condition, also in the Euclidian case, as will be
illustrated by the example of Section \ref{PSD:sec}. On the other
hand, contrarily to the Euclidian case, one could object that the
user must provide at each step an upper bound on a function of
$D(w,v)$, where $v$ is the point appearing in Assumption 1, which
requires some knowledge of  $v$.  This can appear to be a
limitation, but in fact  finding a point $v$ fulfilling Assumption 1
may be quite obvious in practice, and may be far from requiring
direct knowledge of  the point the algorithm is supposed to converge
to, as  illustrated by the example of Section \ref{poincare:sec}.

\begin{proof} The following proof builds upon the  Euclidian case
\cite{bottou-98}.  We are first going to prove that the trajectories asymptotically remain in a compact set. Theorem \ref{thm1} will then easily apply. A second order Taylor expansion yields
 \begin{equation}\begin{aligned}\label{taylor:eq}
D(w_{t+1},v)-D(w_t,v)\leq &2\frac{\gamma_t}{f(w_t)}\langle H(z_t,w_t),\exp_{w_t}^{-1}(v)\rangle \\\quad &+( \frac{\gamma_t}{f(w_t)})^2\norm{H(z_t,w_t)}^2k_1
\end{aligned}\end{equation}
where $k_1$ is an upper bound on the operator norm of half of the Riemannian hessian of  $D(\cdot,v)$ along the geodesic joining  $w_t$ to $w_{t+1}$ (see the Appendix). If the sectional curvature is  bounded from below by $\kappa<0$ we have (\cite{cordero} Lemma 3.12)
\begin{align*}\lambda_{\text{max}}\bigl(\nabla^2_w(D(w,v)/2)\bigr)\leq\frac{\sqrt{|\kappa|D(w,v)}}{\tanh(\sqrt{|\kappa|D(w,v)})}
\end{align*}
where $\nabla^2_w(D(w,v)/2)$ is the Hessian of the squared half distance and $\lambda_{\text{max}}(\cdot)$ denotes the largest eigenvalue of an operator. This implies  that $\lambda_{\text{max}}\bigl(\nabla^2_w(D(w,v)/2)\bigr)\leq \sqrt{|\kappa|D(w,v)}+1$. Moreover,  along the geodesic linking $w_t$ and $w_{t+1}$,  triangle inequality implies $\sqrt{D(w,v)}\leq \sqrt{D(w_t,v)}+\norm{H(z_t,w_t)}$ as  $f(w_t)\geq 1$ and there exists $t_0$ such that $\gamma_t\leq 1$ for $t\geq t_0$. Thus  $k_1\leq \beta(z_t,w_t)$ for $t\geq t_0$ where $\beta(z_t,w_t)= 1+\sqrt{\abs{\kappa}}(\sqrt{D(w_t,v)}+\norm{H(z_t,w_t)})$.     Let  $\mathcal F_t$ be the increasing sequence of  $\sigma$-algebras generated by the several variables available just before time $t$:
 $
\mathcal F_t=\{z_0,\cdots,z_{t-1}\}
$. As $z_t$ is independent from $\mathcal F_t$, and $w_t$ is $\mathcal F_t$ measurable, we have
$ \mathbb E[(\frac{\gamma_t}{f(w_t)})^2\norm{H(z_t,w_t)}^2k_1|\mathcal F_t] \leq
(\frac{\gamma_t}{f(w_t)})^2\mathbb E_z\bigl(
\norm{H(z,w_t)}^2\beta(z,w_t)\bigr) $.
Conditioning \eqref{taylor:eq} to
$\mathcal F_t$, and using Assumption 3:
\begin{equation}\begin{aligned}\label{borne1:eq}
&\mathbb E[D(w_{t+1},v)-D(w_t,v)|\mathcal F_t]\\\quad &\leq 2\frac{\gamma_t}{f(w_t)}\langle \nabla C(w_t),\exp_{w_t}^{-1}(v)\rangle +\gamma_t^2\end{aligned}\end{equation}
Let $\phi:\mathbb R^+\to\mathbb R^+$ be a smooth function such that
\begin{itemize}
\item $\phi(x)=0$ for $0\leq x\leq S$
\item $0< \phi''(x)\leq 2$ for $S< x\leq S+1$
\item $\phi'(x)=1$ for $x\geq S+1$
\end{itemize}
and let $h_t=\phi(D(w_t,v))$.
Let us prove it converges a.s. to $0$. As $\phi''(x)\leq 2$ for all $x\geq 0$ a second order Taylor expansion on $\phi$ yields
\begin{align*}
h_{t+1}-h_t&\leq [D(w_{t+1},v)-D(w_t,v)]\phi'(D(w_t,v))\\\quad &+(D(w_{t+1},v)-D(w_t,v))^2
\end{align*}
Because of the triangle inequality we have $d(w_{t+1},v)\leq d(w_t,v)+\frac{\gamma_t}{f(w_t)} \norm{H(z_t,w_t)}$. Thus $D(w_{t+1},v)-D(w_t,v)\leq 2d(w_t,v)\frac{\gamma_t}{f(w_t)} \norm{H(z_t,w_t)}+(\frac{\gamma_t}{f(w_t)})^2 \norm{H(z_t,w_t)}^2$ which is less than $ \frac{\gamma_t}{f(w_t)}[2d(w_t,v) \norm{H(z_t,w_t)}+ \norm{H(z_t,w_t)}^2] $ for $t\geq t_0$.  Using Assumption 3 and the fact that $w_t$ is measurable $\mathcal F_t$ we have $\mathbb E[h_{t+1}-h_t|\mathcal F_t]\leq \phi'(D(w_t,v)) \mathbb E[D(w_{t+1},v)-D(w_t,v)|\mathcal F_t]+\gamma_t^2$. Using \eqref{borne1:eq} we have
\begin{equation}\begin{aligned}\label{eeee:eq}
 &\mathbb E[h_{t+1}-h_t|\mathcal F_t]\\\quad &\leq 2\frac{\gamma_t}{f(w_t)}\langle \nabla C(w_t),\exp_{w_t}^{-1}(v)\rangle \phi'(D(w_t,v)) +2\gamma_t^2
\end{aligned}\end{equation}
as $\phi'$ is positive, and less than 1.
Either $D(w_t,v)\leq S$  and then we have  $\phi'(D(w_t,v))=0$ and thus  $\mathbb  E[h_{t+1}- h_t|\mathcal F_t]\leq 2\gamma_t^2$. Or  $D(w_t,v)>S$, and Assumption 1 ensures $\langle \nabla C(w_t),\exp_{w_t}^{-1}(v)\rangle$ is negative. As $\phi'\geq 0$,  \eqref{eeee:eq}  implies $\mathbb E[h_{t+1}- h_t|\mathcal F_t]\leq 2\gamma_t^2$.
In both cases $\mathbb E[h_{t+1}- h_t|\mathcal F_t]\leq 2\gamma_t^2$, proving  $h_{t}+2\sum_t^\infty\gamma_k^2$ is a positive supermartingale, hence it converges a.s. Let us prove it necessarily converges to $0$. We have    $\sum_{t_0}^\infty
\mathbb E([\mathbb E(h_{t+1}-h_t|\mathcal F_t)]^+)\leq 2\sum_t\gamma_t^2<\infty $.  Proposition \ref{fisk:prop} proves that $h_t$ is a quasi-martingale.
 Using \eqref{eeee:eq}  we have inequality $$
-2\sum_{t_0}^\infty\frac{\gamma_t}{f(w_t)}\langle \nabla C(w_t),\exp_{w_t}^{-1}(v)\rangle\phi'(D(w_t,v))\leq 2\sum_{t_0}^\infty\gamma_t^2-\sum_{t_0}^\infty\mathbb E[h_{t+1}-h_t|\mathcal F_t]$$and as $h_t$ is a  quasi-martingale we have a.s.
\begin{equation}
\begin{aligned}\label{borne2:eq}&-\sum_{t_0}^\infty\frac{\gamma_t}{f(w_t)}\langle \nabla C(w_t),\exp_{w_t}^{-1}(v)\rangle\phi'(D(w_t,v))\\&\quad\leq 2|\sum_{t_0}^\infty\gamma_t^2|+\sum_{t_0}^\infty|\mathbb E[h_{t+1}-h_t|\mathcal F_t]|<\infty\end{aligned}
\end{equation}
Consider a sample trajectory for which $h_t$ converges to $\alpha>0$.  It means that for $t$ large enough $D(w_t,v)> S$ and thus $\phi'(D(w_t,v))>\epsilon_1>0$. Because of Assumption 1  we have also $\langle\nabla C(w_t),\exp_{w_t}^{-1}(v)\rangle<-\epsilon_2<0$. This contradicts \eqref{borne2:eq} as $\sum_{t_0}^\infty\frac{\gamma_t}{f(w_t)}=\infty$. The last equality comes from \eqref{step:eq} and the fact that   $f$ is continuous and thus bounded along the trajectory.

It has been proved that almost every trajectory asymptotically enters the ball of center $v$ and radius $S$ and stays inside of it. Let us prove that we can work on a fixed compact set.  Let
  $G_n=\bigcap_{t>n}\{ D(w_t,v) \leq S\}$. We have just proved $P(\cup ~G_n)=1$. Thus to prove a.s. convergence, it is thus sufficient to prove a.s.  convergence on each of those sets.  We assume from now on the trajectories all belong to the ball of center $v$ and radius $S$.  As this is a compact set, all continuous functions of the parameter can be bounded. In particular   $\gamma_t/k_3\leq\gamma_t/f(w_t)\leq\gamma_t$ for some $k_3>0$ and thus the modified step size verifies the conditions of Theorem \ref{thm1}. Moreover, $\mathbb E_z(\norm{H(z,w)}^2)\leq A^2$ for some $A>0$ on the compact as it is dominated by $f(w)^2$. As there is no cut locus, this weaker condition is sufficient, since it implies that \eqref{inee:eq} holds. The proof follows from a mere application of Theorem \ref{thm1} on this compact set.

 \end{proof}

Note that, it would be possible to derive an analogous result when a retraction is used instead of the exponential, using the ideas of the proof of   Theorem \ref{thm3}. However, due to a lack of relevant examples,  this result is not presented.

\section{Examples}\label{examples:sec}

Four application examples  are presented.  The first two examples
are rather tutorial. The first one illustrates Theorems \ref{thm1}
and \ref{thm3}. The second one allows to   provide a graphical
interpretation of Theorem \ref{thm2} and its assumptions.   The
third and fourth examples are more detailed and include numerical
experiments. Throughout this section $\gamma_t$ is a sequence of
positive step sizes satisfying the usual condition \eqref{step:eq}.

\subsection{Subspace tracking}\label{grass:sec}

We propose to first revisit in the light of the preceding results
the well-known subspace tracking algorithm of Oja \cite{oja-92}
which is a generalization of the  power method for computing the
dominant eigenvector. In several applications, one  wants to compute
the $r$ principal eigenvectors, i.e. perform principal component
analysis (PCA) of a $n\times n$ covariance matrix $A$, where $r\leq
n$. Furthermore, for computational reasons or for adaptiveness, the
measurements are supposed to be  a stream of $n$-dimensional data
vectors $z_1,\cdots z_t,\cdots$ where $E(z_tz_t^T)=A$ (online
estimation). The problem boils down to estimating  an element of the
Grassmann manifold {Gr}$(r,n)$ of $r$-dimensional subspaces in a
$n$-dimensional ambient space, which can be identified to the set of
rank $r$ projectors:
$$
\text{Gr}(r,n)=\set{P\in\RR^{n\times n}}{P^T=P,~P^2=P,~\tr{P}=r}.
$$
Those projectors can be represented by  matrices $WW^T$ where $W$ belongs to the Stiefel manifold $\text{St}(r,n)$, i.e., matrices of $\RR^{n\times r}$ whose columns are orthonormal.  Define the cost function
$$
C(W)=-\frac{1}{2} \mathbb E_z [z^TW^T Wz]=-\frac{1}{2} \tr{W^TA W}
$$which is minimal when $W$ is a basis of the dominant subspace of the covariance matrix $A$.
It is invariant to rotations $W\mapsto WO, O\in\mathcal O(r)$. The state-space is therefore the set of equivalence classes
$
[W]=\set{WO}{O\in\mathcal O(r)}
$.
This set is denoted by $St(r,n)/\mathcal O(r)$. It is a \emph{quotient representation} of the Grassmann manifold Gr$(r,d)$. This quotient geometry  has been well-studied in e.g. \cite{edelman-98}.   The Riemannian gradient under the event $z$ is:
 $
H(z,W)=(I-W W^T) zz^TW
$. We have the following result
\begin{prop}
Suppose $z_1,z_2,\cdots$ are uniformly bounded. Consider the stochastic Riemannian gradient algorithm
\begin{align}\label{geodes:oja}
W_{t+1}=W_tV_t\cos{(\gamma_t\Theta_t)}V_t^T+U_t\sin{(\gamma_t\Theta_t)}V_t^T
\end{align}
where $U_t\Theta_tV_t$ is the compact SVD of the matrix $(I-W_t W_t^T) z_tz_t^TW_t$. Then $W_t$ converges a.s. to an invariant subspace of the covariance matrix $A$.
\end{prop}
\begin{proof}
The proof is a straightforward application of Theorem \ref{thm1}. Indeed, the update \eqref{geodes:oja} corresponds to \eqref{manifold:grad:eq} as it states that $W_{t+1}$ is on the geodesic emanating from $W_t$ with tangent vector $H(z_t,W_t)$ at a distance $\gamma_t\norm{H(z_t,W_t)}$ from $W_t$.  As the input sequence is bounded, so is the sequence of gradients.
The injectivity radius of the Grassmann manifold is $\pi/2$, and is thus bounded away from zero, and the Grassmann manifold is compact. Thus Theorem \ref{thm1} proves that $W_t$ a.s. converges to a point such that $\nabla C(W)=0$, i.e. $AW=W W^T AW$.
For such points there exists $M$ such that $AW=WM$,  proving $W$  is
an invariant subspace of $A$. A local analysis proves the dominant
subspace of $A$ (i.e. the subspace associated with the first $r$ eigenvalues) is the only stable subspace of the averaged algorithm
\cite{oja-92} under basic assumptions.
\end{proof}

We also have the following result
\begin{prop}
Consider a twice differentiable retraction $R_W$. The algorithm \begin{align}\label{oja3:eq}
W_{t+1}=R_{W_t}\bigl(W_t+\gamma_t(I-W_t W_t^T) z_tz_t^TW_t\bigr)
\end{align} converges a.s. to an invariant subspace of the covariance matrix $A$.
\end{prop}
The result is a mere application of  Theorem \ref{thm3}. Consider in particular the following retraction:  $R_W(\gamma H)$=qf$(W+\gamma H)$ where qf() extracts the orthogonal factor in the QR decomposition of its argument. For small $\gamma_t$, this retraction amounts to follow the gradient in the Euclidian ambient space $\RR^{n\times p}$, and then to orthonormalize the matrix at each step. It is  an infinitely  differentiable retraction \cite{absil-book}. The algorithm  \eqref{oja3:eq} with this particular retraction is known as  Oja's
vector field for subspace tracking and has already been proved to converge in \cite{oja-83}.  Using the general framework proposed in the present paper, we see this convergence result  directly stems from Theorem \ref{thm3}.

This example clearly illustrates the benefits of using a retraction. Indeed, from a numerical viewpoint,  the geodesic update \eqref{geodes:oja}  requires to perform a SVD at each time step, i.e. $O(nr^2)+O(r^3)$ operations,  whereas update  \eqref{oja3:eq} is only an orthonormalization of the vectors having a lower computational cost of order $O(nr^2)$, which can be very advantageous, especially when $r$ is large.

 \subsection{Randomized computation of a Karcher mean on a hyperbolic space}\label{poincare:sec}

We propose to illustrate  Theorem \ref{thm2} and the  assumptions it relies on on a well-known and tutorial manifold. Consider the unit disk $\mathcal D=\{x\in\RR^2 : \norm{x}<1\}$ with the Riemannian metric defined on the tangent plane at $x$ by
$$
\langle\xi,\eta\rangle_x=4\frac{\xi\cdot\eta}{(1-\norm{x}^2)^2}
$$
where $``\cdot"$ represents the conventional scalar product in $\RR^2$. The metric tensor is thus diagonal, so the angles between two intersecting curves in the Riemannian metric are the same as in the Euclidian space. However, the distances differ: as a point is moving closer to the boundary of the disk, the distances are dilated so that the boundary can not be reached in finite time. As illustrated on the figure, the geodesics are either arcs of circles that are orthogonal to the boundary circle, or diameters. The $\Pp$ equipped  with its metric is a Hadamard manifold.

\begin{figure}
\centering
\includegraphics[width=.4\textwidth]{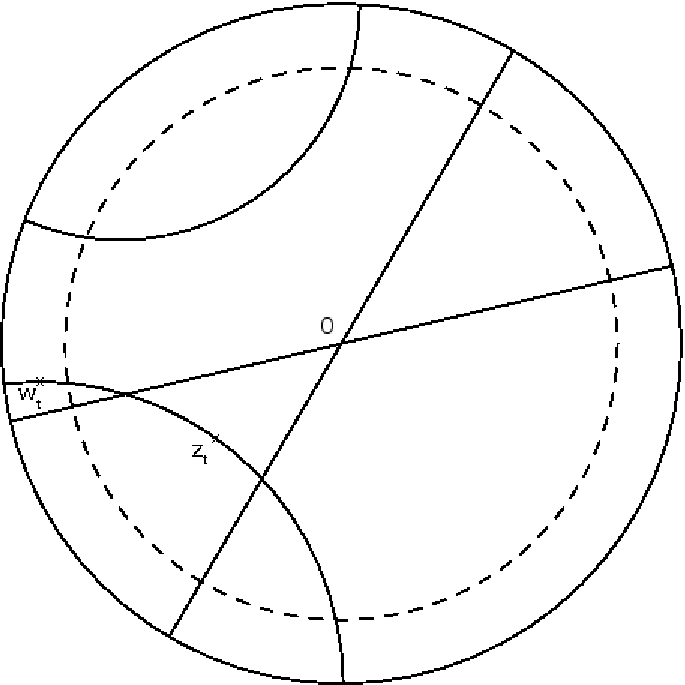}
\caption{The Poincar\'e disk. The boundary is at infinite distance from the center. The geodesics (solid lines) are either  arcs of circles perpendicular to the boundary of the disk, or diameters.  The dashed circle is the boundary of a geodesic ball centered at 0. Assumption 1 is obviously verified: if a point $w_t$ outside the ball makes a small move towards any point  $z_t$ inside the ball along the geodesic linking them, its distance to 0 decreases. } \label{fig1}
\end{figure}

The Karcher (or Fr\'echet) mean on a Riemannian manifold is defined
as the minimizer of $w\mapsto\sum_1^N d^2(w,z_i)$. It can be viewed
as a natural extension of the usual Euclidian barycenter to the
Riemannian case. It is intrinsically defined,   and it is unique on
Hadamard manifolds. There has been growing interest in computing
Karcher means recently,  in particular for filtering on manifolds,
see e.g.  \cite{afsari,barbaresco,arnaudon-10}.  On the $\Pp$   we
propose to compute the mean of $N$ points  in a randomized way. The
method is as follows, and is closely related to the approach \cite{arnaudon-10}. Let $w_t$ be the optimization parameter. The
goal is to find the minimum of the cost function
$$
C(w)=\frac{1}{2N}\sum_1^N d^2(w,z_i)
$$At each time step, a point $z_i$ is randomly picked with an uniform probability law. The loss function is $Q(z_t,w_t)= \frac{1}{2}d^2(w_t,z_i)$, and $H(z_t,w_t)$ is the Riemannian gradient of the half squared geodesic distance $\frac{1}{2}D(z_t,w_t)=\frac{1}{2}d^2(z_t,w_t)$. On the $\Pp$, the distance function is defined by  $d(z,w)=\cosh^{-1}{(1+\delta(z,w))}$ where
$
\delta(z,w)=2\frac{\norm{z-w}^2}{(1-\norm{z}^2)(1-\norm{w}^2)}
$. As the metric tensor is diagonal, the Riemannian gradient and the Euclidian gradient have the same direction. Moreover its norm is simply $d(z_t,w_t)$ (see the Appendix). It is thus easy to prove that
\begin{align}\label{H:eq}
H(z_t,w_t)=\frac{(1-\norm{w_t}^2)(w_t-z_t)+\norm{w_t-z_t}^2 w_t}{\norm{(1-\norm{w_t}^2)(w_t-z_t)+\norm{w_t-z_t}^2 w_t}}d(z_t,w_t)
\end{align}

When there is a lot of redundancy in the data, i.e. when some points are very close to each other,  a randomized algorithm may be much more efficient numerically than a batch algorithm. This becomes obvious in the extreme case where the $z_i$'s are all equal. In this case, the approximated gradient $H(z_t,w_t)$ coincides with the (Riemannian) gradient of the cost $C(w_t)$. However, computing this latter quantity requires $N$ times more operations than computing the approximated gradient. When $N$ is large and when there is a lot of redundancy in the data, we thus see a randomized algorithm can lead to a drastic reduction in the computational burden.  Besides,  note that, the stochastic algorithm can also be used in order to filter  a stream of noisy measurements of a single point on the manifold (and thus  track this point in case it slowly moves). Indeed, it is easily seen that if $\mathcal M=\RR^n$ and $d$ is the Euclidian distance, the proposed update boils down to a first order discrete low pass filter as it computes a weighted mean between the current update $w_t$ and the new measurement $z_t$.
\begin{prop}  Suppose at each time a point $z_t$ is randomly drawn. Let $S>0$ be such that $ S>(\max\{d(z_1,0),\cdots,d(z_N,0)\})^2$  and let $\alpha(w_t)=d(w_t,0)+\sqrt S$. Consider the  algorithm \eqref{manifoldf:grad:eq} where $H(z_t,w_t)$ is given by \eqref{H:eq} and
\begin{align*}f(w_t)^2= \max\{&1, \\&\alpha(w_t)^2(1+d(w_t,0)+ \alpha(w_t)),\\&(2\alpha(w_t)d (w_t,0)+\alpha(w_t)^2)^2\}\end{align*} Then $w_t$  converges  a.s. to the Karcher mean of the points $z_1,\cdots,z_N$.
\end{prop}
\begin{proof}The conditions of Theorem \ref{thm2} are easily checked.  Assumption 1: it is easy to see on the figure that Assumption 1 is  verified  with $v=0$, and $S$ being the radius of an open geodesic ball centered at $0$ and containing all the points $z_1,\cdots,z_N$. More technically, suppose $d(w,0)>\sqrt S$. The quantity $\langle\exp_w^{-1}(0),H(z_i,w)\rangle_w$ is equal to $-d(w,0) H(z_i,w)\cdot w/(1-\norm{w}^2)^2=-\lambda((1-\norm{w}^2)(\norm{w}^2-z_i\cdot w)+\norm{w-z_i}^2 \norm{w}^2)$ where $\lambda$ is a positive quantity bounded away from zero for $d(w,0)>\sqrt S$. As  there exists $\beta>0$ such that $\norm{w}-\norm{z_i}\geq\beta$, and $\norm{w-z_i}\geq \norm{w}-\norm{z_i}$, the term $\langle\exp_w^{-1}(0),H(z_i,w)\rangle_w$ is negative and bounded away from zero, and so is its average over the $z_i'$s.
 Assumption 2: in dimension 2, the sectional curvature is known to be identically equal to $-1$.
 Assumption 3 is obviously satisfied as   $\norm{H(z,w)}=d(z,w)\leq d(z,0)+d(0,w)\leq \sqrt S+d(0,w)=\alpha(w)$ by triangle inequality.

\end{proof}
Note that, one could object that in general finding a function
$f(w_t)$ satisfying Assumption 3 of Theorem \ref{thm2} requires
knowing $d(w_t,v)$ and thus requires some knowledge of the point $v$
of Assumption 1. However, we claim (without proof) that in
applications on Hadamard manifolds, there may be obvious choices for
$v$. This is the case in the present example where finding  $v$ such
that assumptions 1-3 are satisfied requires very little (or no)
knowledge on the point the algorithm is supposed to converge to.
Indeed, $v=0$ is a straightforward choice that always fulfills the
assumptions. This choice is convenient for calculations as the
geodesics emanating from $0$ are radiuses of the disk, but many
other choices would have been possible.

\subsection{Identification of a fixed rank symmetric positive semi-definite matrix}\label{PSD:sec}

To illustrate the benefits of the approach on a recent non-linear  problem, we focus in this section on an algorithm of \cite{meyer-11}, and we prove new rigorous convergence results.  Least Mean Squares (LMS) filters have been extensively utilized in adaptive filtering for online regression.
Let ${x}_t\in\RR^n$ be the input vector, and $y_t$ be the output
defined by $
y_t=w^T{x}_t+\nu_t
$
where the unknown vector $w\in\RR^n$ is to be identified (filter weights), and $\nu_t$ is a noise. At each step we let $z_t=(x_t,y_t)$ and the approximated cost function is $
Q(z_t,w_t)= \frac{1}{2} (w^T x_t-y_t)^2$.
Applying the steepest descent leads to the   stochastic gradient algorithm known as LMS: $
    {w}_{t+1} = w_t-\gamma_t ({w_t}^T{x}_t - y_t) {x}_t$.

We now consider a non-linear generalization of this problem  coming from the  machine learning field (see e.g. \cite{ratsch}), where $x_t\in\RR^{n}$ is the input,  $y_t\in\RR$ is the output, and the matrix counterpart of the linear model  is
\begin{align}\label{matrix:lin:eq}
y_t=\tr{W x_t x_t^T}= x_t^TW x_t
\end{align}
where $W\in\RR^{n\times n}$ is an unknown positive semi-definite matrix to be identified.  In data mining, positive semi-definite matrices $W$ represent a kernel or a Mahalanobis distance, i.e. $W_{ij}$ is the scalar product, or the distance, between instances $i$ and $j$. We assume at each step an expert provides an estimation of  $W_{ij}$ which can be viewed as a  random output.  The goal is to  estimate the matrix $W$ online.  We let $z_t=(x_t,y_t)$ and we will apply our stochastic gradient method to the cost function $C(W)=\mathbb E_z Q(z,W)$  where $
Q(z_t,W_t)= \frac{1}{2} (x_t^TW_t x_t-y_t)^2= \frac{1}{2} (\hat y_t-y_t)^2$.

Due to the large amount of data available nowadays, matrix classification algorithms tend to be applied to computational problems of
ever-increasing size. Yet, they need  be adapted to remain tractable, and the matrices' dimensions need to be reduced so that the matrices are storable.  A wide-spread topical method consists of working with low-rank approximations. Any rank $r$ approximation
of a positive definite matrix can be factored as $A=G G^T$ where $G \in \RR^{n \times r}$. It is then
 greatly reduced in size if $r\ll n$,
leading to a reduction of the numerical cost of typical matrix operations from $O(n^3)$ to  $O(nr^2)$, i.e. linear complexity. This fact has motivated the development of low-rank kernel and Mahalanobis distance learning \cite{kulis-06}, and geometric understanding of the set of  semidefinite positive matrices of fixed rank:
\begin{equation*}
    S_+(r,n)=\set{W\in\mathbb{R}^{n\times n}}{W=W^T\succeq 0,\text{rank}(W)=r}.
\end{equation*}

\subsubsection{Proposed algorithm and convergence results}

To endow $S_+(r,n)$ with a metric we start from the square-root factorization $
     {W} =  {G} {G}^T,
$
where  ${G}\in\mathbb{R}_{*}^{n\times r}$,  i.e. has rank $r$. Because the factorization is invariant by rotation, the search space is   identified to the quotient
   $\FixedRank\simeq\mathbb{R}_{*}^{n\times r}/\OG{r}$,
which represents the set of equivalence classes
\begin{equation*}
    [ {G}] = \set{ {G} {O}}{ {O}\in\OG{r}}.
\end{equation*}
The Euclidian metric
$
    \bar{g}_{ {G}}( {\Delta}_1, {\Delta}_2) = \tr{{\Delta}_1^T {\Delta}_2}$, for $\Delta_1,\Delta_2\in\RR^{n\times r}$ tangent vectors at $G$, is invariant  along the equivalence classes. It thus induces a well-defined metric  $
    g_{[{G}]}(\xi,\zeta) $ on the quotient, i.e. for $\xi,\zeta$ tangent vectors at $[G]$ in $\FixedRank$. Classically \cite{absil-book}, the tangent vectors of the quotient space $\FixedRank$ are identified to the projection onto the horizontal space (the orthogonal space to $[G]$) of tangent vectors of the total space $\RR_*^{n\times r}$.  So tangent vectors at $[ {G}]$ are represented by the set of horizontal tangent vectors
$\{ \Sym( {\Delta}) {G},  {\Delta}\in \mathbb{R}^{n\times n}\}$,
where  $\Sym(A)=(A+A^T)/2$.
The horizontal gradient of $Q(z_t,G_t)$
is the unique horizontal vector $H(z_t,G_t)$ that satisfies the definition of the Riemannian gradient. In the sequel we will systematically identify an element $G$ to its equivalence class $[G]$, which is a matrix of  $S_+(r,n)$.  For more details on this manifold see \cite{journee}.
Elementary computations yield $
H(z_t,G_t) = 2(\hat{y}_t - y_t) {x_t}{x_t}^T {G_t}
$, and \eqref{manifoldf:grad:eq}  writes
\begin{equation}\label{eq:batch-flat}
    G_{t+1} =G_t -\frac{\gamma_t}{f(G_t)}  (\norm{G_t^Tx_t}^2 - y_t) x_t x_t^T{G_t}
\end{equation}where we choose  $f(G_t)=\max(1,\norm{G_t}^6)$ and where the sequence $(\gamma_t)_{t\geq 0}$ satisfies condition \eqref{step:eq}. This non-linear algorithm is   well-defined on the set of equivalence classes, and automatically enforces the rank and positive semi-definiteness constraints of the parameter $G_tG_t^T=W_t$.

\begin{prop}\label{mtns:prop} Let $(x_t)_{t\geq 0}$ be a sequence of zero centered  random vectors of $\RR^n$ with independent identically and normally distributed  components. Suppose $y_t=x_t^TVx_t$ is generated by some  unknown parameter $V\in\FixedRank$. The Riemannian gradient descent \eqref{eq:batch-flat}  is such that $G_tG_t^T=W_t\to W_\infty$ and $\nabla C(W_t)\to 0$ a.s.  Moreover, if $W_\infty$ has rank $r=n$ necessarily $W_\infty= V$ a.s. If $r<n$, necessarily $W$ a.s. converges to an invariant subspace of $V$ of dimension $r$. If this is the dominant subspace of $V$, then $W_\infty=V$.
\end{prop}

The last proposition can be completed with the following fact: it can be  easily proved that the dominant subspace of $V$ is a stable equilibrium of the  averaged algorithm. As concerns for the other invariant subspaces of $V$, simulations indicate they are unstable equilibria. The convergence to $V$
 is thus always expected in simulations.
\begin{proof}
As here the Euclidian gradient of the loss with respect to the
parameter $G_t$ coincides with its projection onto the  horizontal
space $H(z_t,G_t)$, and is thus a tangent vector to the manifold, we
propose to apply Theorem \ref{thm2} to the Euclidian space
$\RR^{n\times r}$, which is of course a Hadamard manifold. This is a
simple way to avoid to compute the sectional curvature of
$\FixedRank$. Note that,  the adaptive step $f(G_t)$ introduced in
Theorem \ref{thm2} and the results of Theorem \ref{thm2} are
nevertheless needed, as the usual assumption  $\mathbb
E_x\norm{H(x,G)}^k\leq A+B\norm{G}^k$ of the Euclidian case is
violated. In fact the proposition can be proved under slightly more
general assumptions: suppose the components of the input vectors
have moments up to the order 8, with second and fourth moments
denoted by $a=\mathbb E(x^i)^2$ and $b=\mathbb E(x^i)^4$ for $1\leq
i\leq n$ such that $b>a^2>0$. We begin with a preliminary result:
\begin{lem}
Consider the linear (matrix) map
 $
U:M\mapsto \mathbb E_x (\tr{ x x^T M}{x}{x}^T)
$.
  $U(M)$ is the matrix whose coordinates are $a^2(M_{ij}+M_{ji})$ for $i\neq j$, and $\tr{M}a^2+M_{ii}(b-a^2)$ for $i=j$.
\end{lem}

Assumption 1:  We let $v=0$. Let ``$\cdot$" denote the usual scalar product in $\RR^{n\times r}$. For $\norm{G}^2$ sufficiently large,  $G\cdot(v-G)=-\tr{[\mathbb E (\tr{ x x^T (GG^T-V)}{x}{x}^TG]G^T)}< -\epsilon<0$, which means the gradient tends to make the norm of the parameter decrease on average, when it is  far from the origin. Indeed let $P=GG^T$. We want to prove that $\tr{U(P)P}> \tr{U(V)P}+\epsilon$ for sufficiently large $\norm{G}^2=\tr{P}$. If we choose a basis in which $P$ is diagonal we have $\tr{U(P)P}=a^2\tr{P}^2+(b-a^2)\tr{P^2}=a^2(\sum\lambda_i)^2+(b-a^2)(\sum\lambda_i^2)$ where $\lambda_1,\dots,\lambda_n$ are the eigenvalues of $P$. We have also $\tr{U(V)P}=a^2\tr{P}\tr{V}+(b-a^2)\sum(\lambda_i V_{ii})$. For sufficiently large $\tr{P}$, $\tr{P}^2$ is arbitrarily larger than $\tr{P}\tr{V}$. We have also $(\sum(\lambda_i^2)\sum(V_{ii}^2))^{1/2}\geq\sum(\lambda_i V_{ii})$ and  $(\sum\lambda_i^2)^{1/2}\geq\frac{1}{n}\tr{P}$ by Cauchy-Schwartz inequality. Thus for    $\tr{P}\geq n\sum(V_{ii}^2)^{1/2}$, we have $(\sum\lambda_i^2)^{1/2}\geq\sum(V_{ii}^2)^{1/2}$ and thus $\sum\lambda_i^2\geq(\sum(\lambda_i^2)\sum(V_{ii}^2))^{1/2}\geq\sum(\lambda_i V_{ii})$.
  Assumption 2 is satisfied as the curvature of an Euclidian space  is zero. For Assumption 3, using the fact that for $P,Q$ positive semi-definite $\tr{PQ}\leq(\tr{P^2}\tr{Q^2})^{1/2}\leq\tr{P}\tr{Q}$, and that $\tr{GG^T}=\norm{G}^2$ and $\tr{xx^T}=\norm{x}^2$, it is easy to prove there exists $B>0$ such that
$ \norm{H(x, G)}^2=(\norm{G^Tx}^2 - y)^2 \norm{x x^T{G}}^2\leq (\norm{G}^6+B\norm{G}^2)\norm{x}^8$. Thus there exists $\mu >0$ such that $[\max(1,\norm{G}^3)]^2$  is greater than $\mu\mathbb E_z \norm{H(x, G)}^2$. On the other hand, there exists $\lambda$ such that $\lambda\mathbb E_z(2\norm{H(x, G)}\norm{ G} +\norm{H(x, G)}^2)^2\leq \max(1,\norm{G}^6)$. But the alternative step $\max(\mu,\lambda)\gamma_t$ satisfies condition \eqref{step:eq}.

Let us analyze the set of possible asymptotic values. It is characterized by $U(GG^T-V)G=0$. Let $M$ be the symmetric matrix $GG^T-V$. If $G$ is invertible, it means  $U(M)=0$. Using the lemma above we see that the off-diagonal terms of $M$ are equal to 0, and summing the diagonal terms $((n-1)a^2+b)\tr{M}=0$ and thus $\tr{M}=0$ which then implies $M=0$ as $b>a^2$.  Now suppose $r<n$. If $b=3a^2$, as for the normal distribution, $U(M)=2M+\tr{M} I$ and $U(GG^T-V)G=0$ implies $G(kI+2G^TG)=2VG$ for some $k\in\RR$. Thus (as in example \ref{grass:sec}) it implies $G$ is an invariant subspace of $V$. Similarly to  the full rank case, it is easy to prove  $W_\infty=V$ when $V$ and $G$ span the same subspace.
\end{proof}

\subsubsection{Gain tuning}

The condition \eqref{step:eq} is common in stochastic approximation. As in the standard filtering problem, or in Kalman filter theory, the more noisy observations of a constant process one gets, the weaker the gain of the filter becomes. It is generally recommended to set
$
\gamma_t={a}/({1+b~t^{1/2+\epsilon}})$
where in theory  $\epsilon>0$, but in practice we propose to take $\epsilon=0$, leading to the family of gains
\begin{align}\label{gain}
\gamma_t=\frac{a}{1+b~t^{1/2}}
\end{align}
If the gain remains too high, the noise will make the estimator oscillate around the solution. But   a  low   gain leads to slow convergence. The coefficient $a$ represents the initial gain. It must be high enough  to ensure sufficiently fast  convergence but not excessive to avoid amplifying the noise. $b$ is concerned with the asymptotic behavior of the algorithm and must be set such that the algorithm is insensitive to noise in the final iterations (a high noise could destabilize the final matrix identified over a given training set). $a$ is generally set experimentally using a reduced number of iterations, and $b$ must be such that the variance of the gradient is very small compared to the entries of $G_t$ for  large $t$.

\subsubsection{Simulation results}

Asymptotic convergence of $GG^T$ to the true value $V$ is always achieved in simulations. When $t$ becomes large,  the behavior of the stochastic algorithm is very close to the behavior of the averaged gradient descent algorithm $J_{t+1}=J_t-\frac{\gamma_t}{f(J_t)}\mathbb E_z H(z,J_t)$, as illustrated in Figure \ref{fig:1}. This latter algorithm has a well characterized behavior in simulations: in a first phase the estimation error decreases rapidly, and in a second phase it slowly converges to zero. As the number of iterations increases, the estimation error becomes arbitrarily small.

In all the experiments, the estimated matrices have an initial norm
equal to $\norm{V}$.  This is because a large initial norm
discrepancy induces a rapid decrease in the estimation error, which
then would seem to tend very quickly to zero compared to its initial
value. Thus, a fair experiment requires initial comparable norms. In
the first set of numerical experiments Gaussian input vectors $
x_t\in\RR^{100}$  with 0 mean and identity covariance matrix are
generated.   The output is generated via model \eqref{matrix:lin:eq}
where $V\in\RR^{100\times 100}$ is a symmetric positive
semi-definite matrix with  rank $r=3$. The results are illustrated
on Figure \ref{fig:1} and indicate the matrix $V$ is asymptotically
well identified.

\begin{figure}
    \centering
    \includegraphics[width=\textwidth]{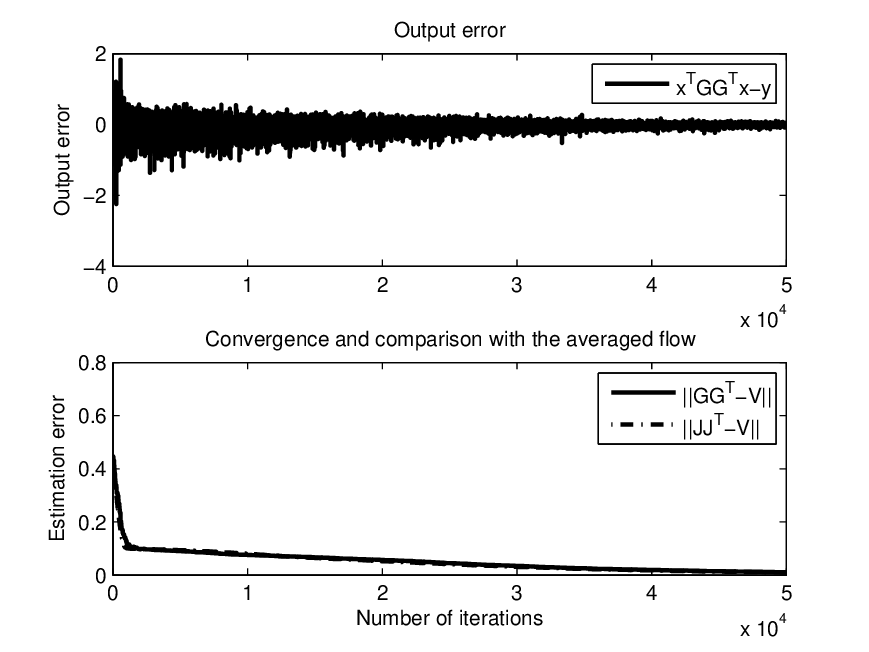}
    \caption{Identification of a rank $3$ matrix $V$ of dimension $100\times 100$ with algorithm \eqref{eq:batch-flat}. Top plot: output (or classification) error versus the number of iterations. Bottom plot: estimation error for the stochastic algorithm $\norm{G_tG_t^T-V}$ (solid line) and estimation error for the deterministic averaged algorithm $J_{t+1}=J_t-\frac{\gamma_t}{f(J_t)}\mathbb E_z H(z,J_t)$ (dashed line). The curves nearly coincide. The chosen gain is $ \gamma_t=.001/(1+t/5000)^{1/2}$.}
    \label{fig:1}
\end{figure}

\begin{figure}
    \centering
    \includegraphics[width=\textwidth]{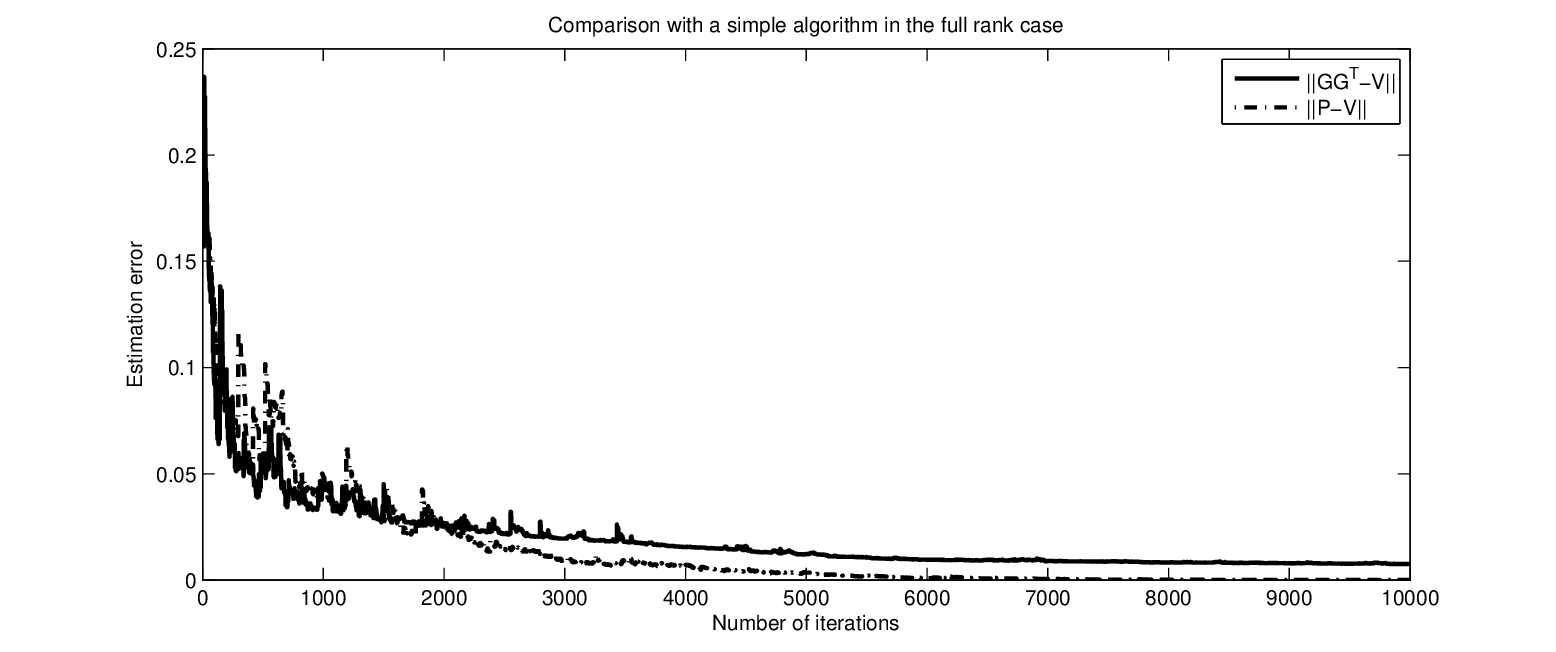}
    \caption{Full-rank case with $r=n=20$. Plot of the estimation error for algorithm \eqref{eq:batch-flat} (solid line) and  \eqref{naive:eq} (dashed line). The gain is $\gamma_t=.01/(1+t/500)^{1/2}$.}
    \label{fig:2}
\end{figure}

In order to compare the proposed method with another algorithm,  we  propose to focus on the full-rank case, and compare the algorithm with a naive but efficient technique. Indeed, when $r=n$,   the cost function $C(W)$ becomes convex in the parameter $W\in P_+(n)$ and the only difficulty is to numerically maintain $W$ as positive semi-definite. Thus, a simple method to attack the problem of identifying $V$ is to derive a stochastic gradient algorithm in $\RR^{n\times n}$, and to project at each step the  iterate on the cone of positive semi-definite matrices, i.e.,\begin{align}P_0\in S_+(n,n),\qquad
P_{t+1}=\pi(P_t-{\gamma_t}\nabla Q(z_t,P_t))\label{naive:eq}
\end{align}where $\pi$ is the projection on the cone. It has been proved in \cite{higham} this projection can be performed by diagonalizing the matrix, and setting all the negative eigenvalues equal to zero.  Figure \ref{fig:2} illustrates the results. Both algorithms have comparable performances. However, the proposed algorithm \eqref{eq:batch-flat} is a little slower than the stochastic algorithm \eqref{naive:eq} which is as expected, since this latter algorithm takes advantage of the convexity of the averaged cost function in the full-rank case.

However, the true advantage of the proposed approach is essentially computational, and becomes more apparent when the rank is low. Indeed, when $n$ is very large and  $r\ll n$ \eqref{eq:batch-flat} has linear complexity in $n$, whereas a method based on diagonalization requires at least $O(n^2)$ operations  and may become intractable. Moreover, the problem is not convex anymore due to the rank constraint and  an approximation then projection technique can lead to degraded performance. Thus, comparing both techniques \eqref{eq:batch-flat} and a technique based on diagonalization such as \eqref{naive:eq} is pointless for low-rank applications, and finding relevant algorithms is an involved task that has been recently addressed in several papers, see e.g.  \cite{meyer-11,shalit}. Since in the present paper  the emphasis is put on mathematical convergence results, the interested reader can refer to \cite{meyer-thesis} where  \eqref{eq:batch-flat} and its variants have been extensively tested on several databases. They are shown to compete with state of the art methods, and to scale very well when the  matrix is of very large dimension (a variant was tested on the Netflix prize database).  They have also  been  recently compared to more involved Riemannian methods in \cite{shalit}.

\subsection{Non-linear gossip algorithm for decentralized covariance matrix estimation}

The problem of computing distributed averages on a network appears in many applications,   such as multi-agent systems, distributed data fusion,
and decentralized optimization. The underlying idea is to replace expansive wiring by a network where   the information exchange is reduced owing to various limitations in data communication. A way to compute distributed averages that has gained popularity over the last years is the so-called gossip algorithm \cite{boyd-gossip}. It is a randomized  procedure where at each iteration, a node communicates with one neighbor and both nodes set their value equal to the average of their current values.  The goal is  for all nodes to reach a common intermediate value as quickly as possible, with little computational power. Gossip algorithms are a special case of distributed optimization algorithms, where stochastic gradient descent plays a central role (see e.g. \cite{tsitsiklis}). When applied to multi-agent systems, this algorithm allows the agents to reach a consensus, i.e. agree on a common quantity. This is the well known consensus problem \cite{moreau}. When the  consensus space is not linear (for instance a group of agents wants to agree on a common direction of motion, or a group of oscillators on a common phase) the methods need to be adapted to the non-linearities of the problem (see e.g. \cite{sepulchre-08}). Consensus on manifolds has recently received  increasing attention, see e.g. the recent work of \cite{tron,sarlette} for deterministic procedures.  In \cite{sarlette-08},  a gossip algorithm on the circle has been proposed and analyzed.

In this section, we address the problem  of estimating a covariance matrix $W$ on a sensor network in a decentralized way (see e.g. \cite{ljung2}): suppose each node $i$ provides measurements $y_i=[y_{i}^1,\cdots,y_{i}^m]\in\RR^{n\times m}$ where the vectors $y_{i}^j\in\RR^n$ are zero centered normally distributed random vectors with a covariance matrix to be estimated. After a local computation, each node is assumed to possess an initial estimated covariance matrix $W_{i,0}$. Neighboring nodes are allowed to exchange information  at random times. We assume the nodes  are labeled according to their proximity as follows: for $i\leq m-1$ the nodes $i$ and $i+1$ are neighbors.  At each time  step $t$, we suppose a node $i<m$ is picked randomly with probability $p_i>0$ (where $p_i$ represents for instance the frequency of availability of the communication channel between nodes $i$ and $i+1$)  and the  nodes in question update their covariance estimates $W_{i,t}$ and $W_{i+1,t}$. Our goal is that they reach consensus on a common intermediate value. To do so,  the procedure \eqref{manifold:grad:eq} is implemented using an interesting alternative  metric  on the cone of positive definite matrices $P_+(n)$.

\subsubsection{A statistically meaningful distance on $P_+(n)$} Information geometry allows to define Riemannian distances  between probability distributions that are very meaningful from a statistical point of view. On the manifold of symmetric positive definite matrices $P_+(n)$, the so-called Fisher Information Metric for two tangent vectors $X_1,X_2$ at $P\in P_+(n)$ is given by
\begin{align}\label{inf:metric}
\langle X_1,X_2\rangle_P=\tr{X_1P^{-1}X_2P^{-1}}
\end{align}
It defines an infinitesimal distance  that agrees with the celebrated Kullback-Leibler divergence between probability distributions: up to third order terms in a small symmetric matrix, say,  $X$ we have
$$
KL(\mathcal N(0,P)||\mathcal N(0,P+X))=\langle X,X\rangle_P
$$
for any $P\in P_+(n)$ where $\mathcal N(0,P)$ denotes  a Gaussian distribution of zero mean and covariance matrix $P$ and $KL$ denotes the Kullback Leibler divergence. The geodesic distance writes
 $
d(P,Q)=(\sum_{k=1}^n\log^2(\lambda_k))^{1/2}
 $ where $\lambda_1,\cdots,\lambda_n$ are the eigenvalues the matrix $PQ^{-1}$ and it represents the amount of information that separates $P$ from $Q$.

The introduced notion of statistical information  is easily understood from a simple variance estimation problem with $n=1$. Indeed, consider the problem of estimating the variance of a random vector $y\simeq\mathcal N(0,\sigma)$. In statistics, the Cramer-Rao bound provides a lower bound on the accuracy of any unbiased estimator $\hat \sigma$ of the variance $\sigma$: here it states $\mathbb E(\hat\sigma-\sigma)^2\geq\sigma^2$. Thus the smaller $\sigma$ is, the more potential information the distribution contains about $\sigma$.  As a result, two samples drawn respectively from, say, the distributions  $\mathcal N(0,1000)$ and $\mathcal N(0,1001)$ look much more similar  than
 samples drawn respectively from the distributions $\mathcal N(0,0.1)$ and $\mathcal N(0,1.1)$. In other words, a unit increment in the variance will have a small impact on the corresponding  distributions if initially $\sigma=1000$ whereas it will have a high impact if $\sigma=0.1$. Identifying zero mean Gaussian distributions with their variances, the Fisher Information Metric accounts for that statistical discrepancy as \eqref{inf:metric}  writes $\langle d\sigma,d\sigma\rangle_\sigma=(d\sigma/\sigma)^2$. But the Euclidian distance does not, as the Euclidian distance between the variances is equal to 1 in both cases.

The metric  \eqref{inf:metric} is also known as the natural metric on $P_+(n)$, and admits strong invariance properties (see Proposition \ref{propinv} below). These properties make the Karcher mean associated with this distance more robust to outliers than the usual arithmetic mean. This is the main reason why this distance has attracted ever increasing attention  in  medical imaging applications (see e.g. \cite{pennec-06}) and radar processing \cite{barbaresco} over the last few years.

\subsubsection{A novel randomized algorithm}

We propose  the following randomized procedure to tackle the problem above. At each step $t$, a node $i<m$ is  picked randomly  with probability $p_i>0$,  and both neighboring nodes $i$ and $i+1$ move their values towards each other along the geodesic linking them, to a distance $\gamma_t d(W_{i,t},W_{i+1,t})$ from their current position. Note that, for $\gamma_t=1/2$ the updated matrix $W_{i,t+1}$ is at exactly half Fisher information (geodesic) distance between $W_{i,t}$ and $W_{i+1,t}$. This is an application of update \eqref{manifold:grad:eq} where $z_t$ denotes the selected node at time $t$ and has probability distribution $(p_1,\cdots,p_{m-1})$, and where the average cost function writes
$$
C(W_1,\cdots,W_m)=\sum_{i=1}^{m-1}~p_i ~d^2(W_i,W_{i+1})
$$on the manifold $P_+(n)\times\cdots \times P_+(n).$ Using the explicit expression of the geodesics \cite{faraut},  update \eqref{manifold:grad:eq}  writes
\begin{equation}\label{goss:eq}
\begin{aligned}
W_{i,t+1}&=W_{i,t}^{1/2}\exp(\gamma_t\log(W_{i,t}^{-1/2}W_{i+1,t}W_{i,t}^{-1/2}))W_{i,t}^{1/2},\\
W_{i+1,t+1}&=W_{i+1,t}^{1/2}\exp(\gamma_t\log(W_{i+1,t}^{-1/2}W_{i,t}W_{i+1,t}^{-1/2}))W_{i+1,t}^{1/2}
\end{aligned}
\end{equation}

This algorithm has several theoretical advantages. First it is based on the Fisher information metric, and thus is natural from a statistical viewpoint. Then it has several nice properties as illustrated by the following two  results:

   \begin{prop}\label{propinv}
   The algorithm \eqref{goss:eq} is invariant to the action of $GL(n)$ on $P_+(n)$ by congruence.
   \end{prop}
   \begin{proof}
   This is merely a  consequence of the invariance of the metric, and of the geodesic distance  $d(GP G^T,GQG^T)=d(P,Q)$ for any $P,Q\in P_+(n)$ and $G\in GL(n)$.
   \end{proof}
The proposition has the following meaning: after a linear change of coordinates,  if all the measurements $y_{i}^j$ are transformed into new measurements $Gy_{i}^j$, where $G$ is an invertible matrix on $\RR^n$, and the corresponding estimated initial covariance matrices are accordingly transformed into $GW_{i,0} G^T$, then the algorithm \eqref{goss:eq} is unchanged, i.e., for any node $i$ the algorithm with initial values $GW_{i,0} G^T$'s will yield  at time $t$ the matrix $GW_{i,t} G^T$, where $W_{i,t}$ is the updated value corresponding to the initial values $W_{i,0}$'s.  As a result, the algorithm will perform equally well for a given problem independently of the choice of coordinates in which the covariance matrices are expressed (this implies in particular invariance to the orientation of the axes, and invariance to change of units, such as meters versus feet, which can be desirable and physically meaningful in some applications). The most important result is  merely an application of the theorems of the present paper:

\begin{prop}If the sequence $\gamma_t$ satisfies the usual assumption \eqref{step:eq}, and is upper bounded by 1/2,  the covariance matrices  at each node converge a.s. to a common value.
\end{prop}
\begin{proof}
This is simply an application of Theorem 1. Indeed, $P_+(n)$ endowed with the natural metric is a complete manifold, and thus one can find a geodesic ball containing the initial values $W_1,\cdots,W_m$. In this manifold, geodesic balls are convex (see e.g. \cite{afsari}). At each time step two points move towards each other along the geodesic linking them, but as $\gamma_t\leq 1/2$ their updated value lies between their current values, so they remain in the ball  by convexity. Thus, the values belong to a compact set at all times. Moreover the injectivity radius is bounded away from zero, and the gradient is bounded in the ball, so Theorem \ref{thm1} can be applied. $\nabla_{W_m}C=0$   implies that $d(W_{m-1},W_m)=0$  as $\nabla_{W_m}C=-2p_{m-1}\exp_{W_m}^{-1}(W_{m-1})$ (see Appendix). Thus $W_{m-1}$ and $W_m$ a.s. converge to the same value. But as $\nabla_{W_{m-1}}C=0$ this implies $W_{m-2}$ converges a.s. to the same value. By the same token we see all nodes converge to the same value a.s.
\end{proof}

\subsubsection{Simulation results}
As the cone of positive definite matrices $P_+(n)$ is convex, the standard gossip algorithm is well defined on $P_+(n)$. If node $i$ is drawn, it simply consists of the update
\begin{equation}\label{goss2:eq}
W_{i,t+1}=W_{i+1,t+1}=(W_{i,t}+W_{i+1,t})/2
\end{equation}
In the following numerical experiments we let  $n=10$, $m=6$, and nodes are drawn with uniform probability. The step $\gamma_t$ is fixed equal to $1/2$ over the experiment  so that \eqref{goss:eq} can be viewed as a Riemannian gossip algorithm (the condition \eqref{step:eq} is only concerned with the asymptotic behavior of $\gamma_t$ and can thus be satisfied even if $\gamma_t$ is fixed over a finite number of iterations). Simulations show that both algorithms always converge. Convergence of the Riemannian algorithm \eqref{goss:eq} is illustrated in Figure \ref{fig5}.

\begin{figure}
    \centering
    \includegraphics[width=\textwidth]{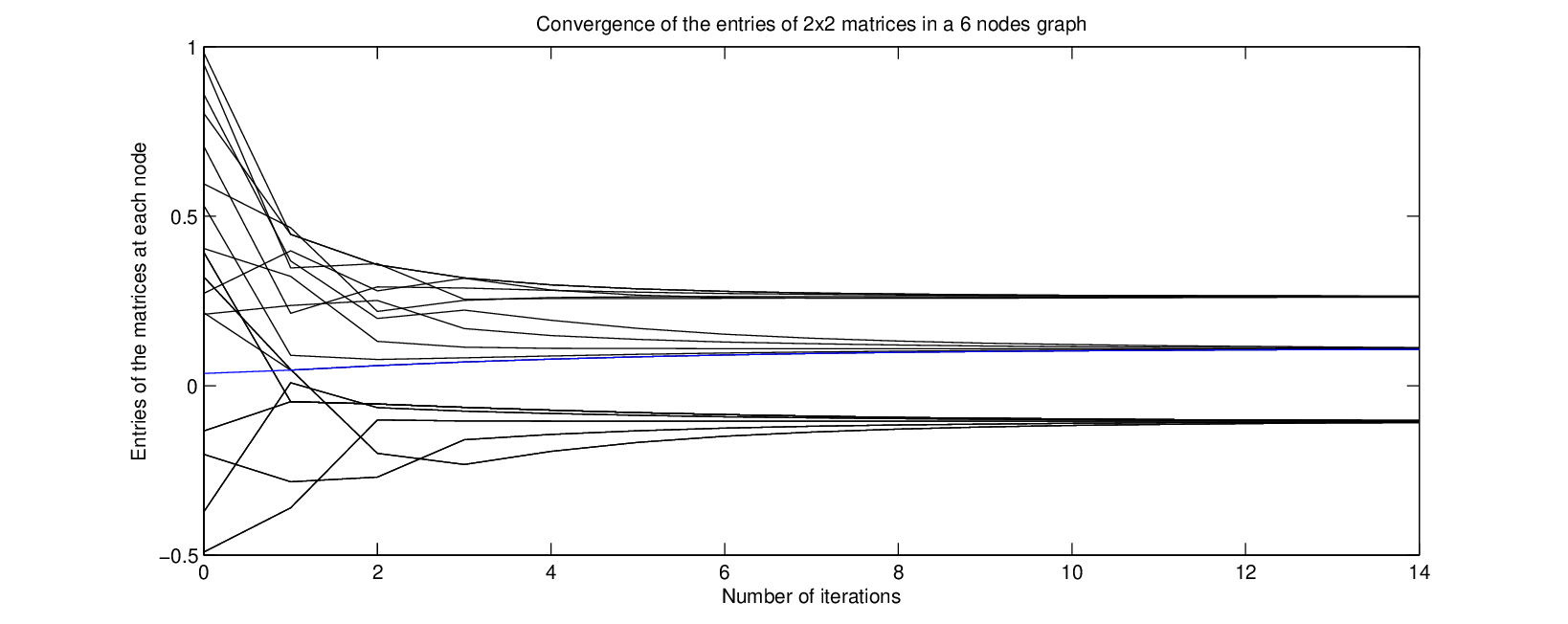}
    \caption{Entries of the matrices at each node versus the number of iterations over a single run. The matrices are of dimension $2\times 2$ and the graph has 6 nodes. Convergence to a common (symmetric) matrix is observed.}  \label{fig5}
\end{figure}

In Figures \ref{fig32} and \ref{fig2}, the two algorithms are compared.  Due to the stochastic nature of the algorithm, the simulation results are averaged over 50 runs. Simulations  show that the Riemannian algorithm \eqref{goss:eq} converges faster in average than the usual gossip algorithm \eqref{goss2:eq}. The convergence is slightly faster when the initial matrices $W_{1,0},\cdots,W_{m,0}$ have approximately  the same norm. But when the initial matrices are far from each other, the Riemannian consensus algorithm outperforms the usual gossip algorithm. In Figure \ref{fig32}, the evolution of the cost $C(W_{1,t},\cdots,W_{m,t})^{1/2}$ is plotted versus the number of iterations. In Figure \ref{fig2},  the diameter of the convex hull of matrices $W_{1,t},\cdots,W_{m,t}$ is considered as an alternative convergence criterion. We see the superiority of the Riemannian algorithm is particularly striking with respect to this convergence criterium. It  can also be  observed in simulations that the Riemannian algorithm is  more robust to outliers. Together with its statistical motivations, and  its invariance and guaranteed convergence properties, it makes it an interesting procedure for  decentralized covariance estimation,  or more generally randomized consensus on $P_+(n)$.

 \begin{figure}[ht!]
    \centering
    \includegraphics[width=\textwidth]{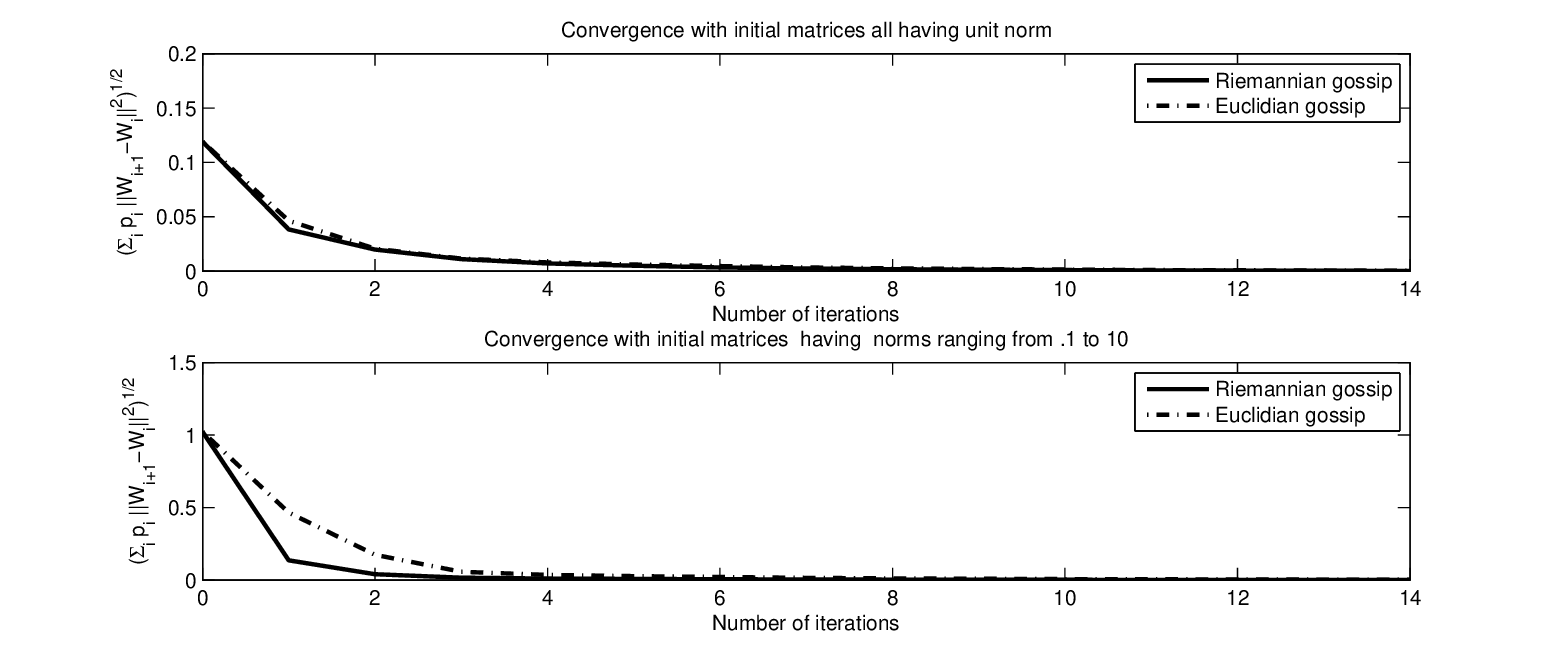}
    \caption{Comparison of Riemannian (solid line) and Euclidian (dashed line) gossip for covariance matrices of dimension $10\times 10$ with a 6 nodes graph. The plotted curves represent the square root of the averaged cost $C(W_{1,t},\cdots,W_{m,t})^{1/2}$   versus the number of iterations, averaged over 50 runs. The Riemannian algorithm converges faster  (top graphics). Its superiority is particularly striking when the nodes have  heterogeneous initial values (bottom graphics).}
    \label{fig32}
\end{figure}

\begin{figure}[ht!]
    \centering
    \includegraphics[width=\textwidth]{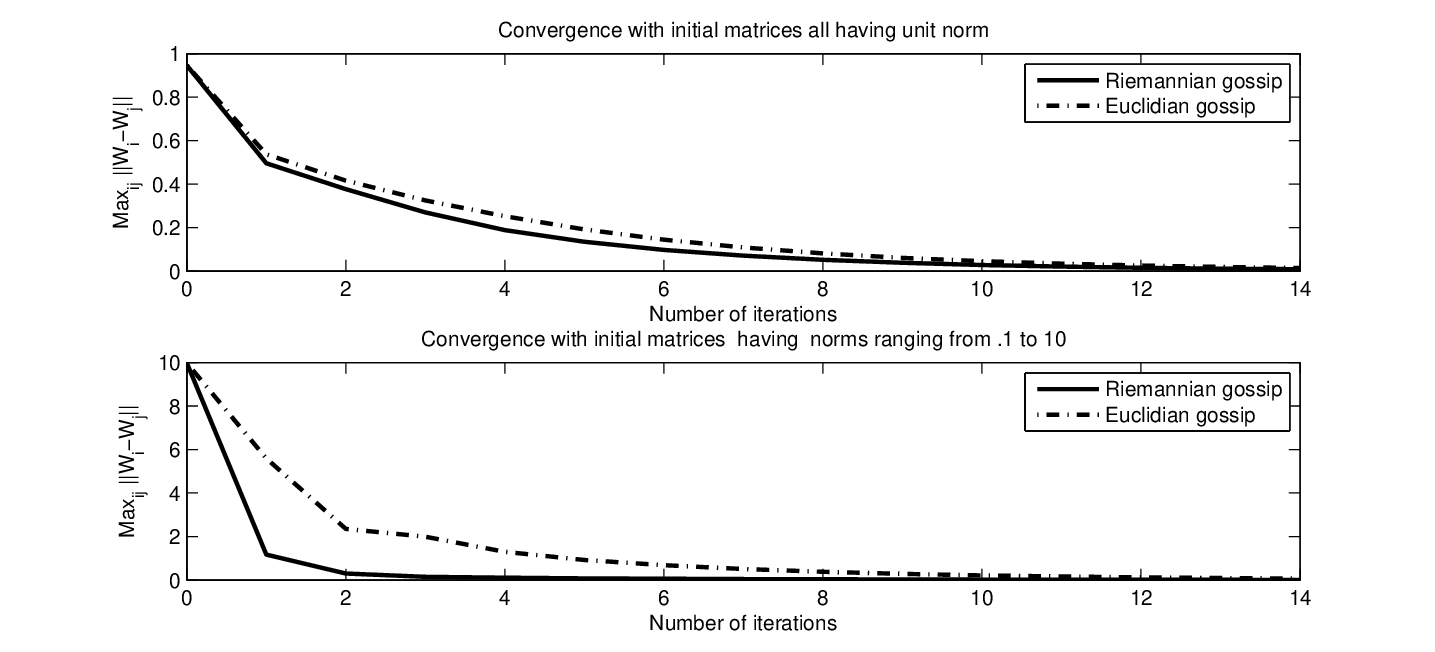}
    \caption{Comparison of Riemannian (solid line) and Euclidian (dashed line) gossip for covariance matrices of dimension $10\times 10$ with a 6 nodes graph with another convergence criterion. The plotted curves represent the diameter of the convex hull $\max_{i,j}\norm{W_{i,t}-W_{j,t}}$   versus the number of iterations, averaged over 50 runs. The Riemannian algorithm converges faster  (top plot). It outperforms the Euclidian algorithm when the nodes have  heterogeneous initial values (bottom plot).}
    \label{fig2}
\end{figure}

\section{Conclusion}

In this paper we proposed a stochastic gradient algorithm on Riemannian manifolds. Under reasonable assumptions the convergence of the algorithm was proved. Moreover the convergence results are proved to hold when a retraction is used, a feature of great practical interest. The approach is versatile, and  potentially applicable to numerous non-linear problems in several fields of research such as control, machine learning, and signal processing, where the manifold  approach is often used either to enforce a constraint, or to derive an intrinsic algorithm.

Another important connection with the literature concerns Amari's natural gradient \cite{amari-98}, a technique that has led to substantial gains in the blind source separation problem, and that can be cast in our framework. Indeed, the idea is to consider successive realizations $z_1,z_2,\cdots$ of a parametric model with parameter $w\in\RR^n$ and joint probability $p(z,w)$. The goal is to estimate online the parameter $w$. Amari proposes to use algorithm  \eqref{manifold:retraction:eq} where the Riemannian metric is the Fisher Information Metric associated to the parametric model,  the loss is the log-likelihood $Q(z,w)=\log p(z,w)$, and the retraction is the mere addition in $\RR^n$. The resulting algorithm, the so-called natural gradient, is proved to be asymptotically efficient, i.e., to reach an asymptotical Cramer-Rao lower bound. Using the true exponential map and thus algorithm  \eqref{manifold:grad:eq} would result in a different (intrinsic) update. In \cite{smith-2005}, S. Smith has proposed an intrinsic Cramer-Rao bound based on the Fisher metric. In future work, one could explore whether the intrinsic algorithm  \eqref{manifold:grad:eq} asymptotically reaches the \emph{intrinsic} Cramer-Rao bound. More details are given in Appendix A.

In the future we would also like to explore two of the aforementioned applications.  First, the matrix completion problem. Proving the convergence of the stochastic gradient algorithm \cite{meyer-icml} requires to study the critical points of the averaged cost function. This leads to prove mathematically involved results on low-rank matrix identifiability, possibly extending the non-trivial work of \cite{candes}. Then, we would like to prove more general results for non-linear consensus on complete manifolds   with a stochastic communication graph. In particular we hope to extend or improve the convergence bounds of the gossip algorithms in the Euclidian case \cite{boyd-gossip} to problems such as the ones described in \cite{sepulchre-08,sarlette-08}, and also to understand to what extent the gossip algorithms for consensus can be faster in a hyperbolic geometry.

\subsection*{Acknowledgements}
The author would like to thank Gilles Meyer and Rodolphe Sepulchre for early collaboration on the subject,  L\'eon Bottou for interesting discussions.

\section*{Appendix A: Links with information geometry}

An important concept in information geometry is the natural gradient. Let us show it is related to the method proposed in this paper.
Suppose now that $z_t$ are realizations of a parametric model with parameter $w\in\RR^n$ and joint probability density function $p(z,w)$. Now let
$$
Q(z,w)=l(z,w)=\log(p(z,w))
$$
be the log-likelihood of the parametric law $p$. If $\hat w$ is an estimator of the true parameter $w^*$ based on $k$ realizations of the process $z_1,\cdots,z_k$ the covariance matrix is larger than the Cramer-Rao bound:
$$
\mathbb E[(\hat w-w^*)(\hat w-w^*)^T]\geq \frac{1}{k} G(w^*)^{-1}
$$
with $G$ the Fisher information matrix $G(w)=-\mathbb E_z [(\nabla_w^El(z,w))(\nabla_w^El(z,w))^T]$ where $\nabla^E$ denotes the conventional gradient in Euclidian spaces. As $G(w)$ is a positive definite matrix it defines a Riemannian structure on the state space $\mathcal M=\RR^n$, known as the Fisher information metric. In this chart the Riemannian gradient of $Q(z,w)$ writes $G^{-1}(w)\nabla_w^E l(z,w)$. As $\mathcal M=\RR^n$, a simple retraction is the addition $R_w(u)=w+u$. Taking $\gamma_t=1/t$ which is compatible with assumption 1, update \eqref{manifold:retraction:eq} writes $w_{t+1}=w_t-\frac{1}{t}G^{-1}(w_t)\nabla_w^E l(z_t,w_t)$. This is the celebrated Amari's natural gradient \cite{amari-98}.

Assuming $w_t$ converges to the true parameter $w^*$ generating the data, Amari
 proves it is an asymptotically efficient estimator. Indeed, letting $V_t=
\mathbb E[(w_t-w^*)(w_t-w^*)^T]$ we have
$$
V_{t+1}=V_t-2\mathbb E[\frac{1}{t}G^{-1}\nabla_w^E l(z_t,w_t)(w_t-w^*)^T]+\frac{1}{t^2}G^{-1}GG^{-1}+O(\frac{1}{t^3})
$$But up to second order terms $\nabla_w^E l(z_t,w_t)=\nabla_w^E l(z_t,w^*)+(\nabla_w^E )^2l(z_t,w^*)(w_t-w^*)$, with $\mathbb E[l(z,w^*)]=0$ as $w^*$ achieves a maximum of the expected log-likelihood where the expectation is with respect to the law $p(z,w^*)$, and  $G(w)=\mathbb E[(\nabla_w^E )^2l(z_t,w)]$ because of the basic properties of the Cramer Rao bound. Finally $
V_{t+1}=V_t-2V_t/t+G^{-1}/t^2$, up to terms  whose average can be neglected. The asymptotic solution of this equation is $V_t=G^{-1}/t+O(1/t^2)$ proving statistical efficiency.

%in ${1}/{t^3},||w_t-w||^3/t,||w_t-w||/t^2$

 It completes our convergence result, proving that when the space is endowed with the Fisher metric, and the trivial retraction is used, the stochastic gradiend method proposed in this paper provides an asymptotically efficient estimator. The natural gradient has been   applied to blind source separation (BSS) and has proved to lead to substantial performance gains.

\cite{smith-2005} has recently derived an intrinsic Cramer-Rao
bound. The bound does not depend on any non-trivial choice of
coordinates, i.e. the estimation error $\norm{\hat w-w}$ is replaced
with the Riemannian distance associated to the Fisher information
metric. In the same way, the usual natural gradient update
$w_{t+1}=w_t-\frac{1}{t}G^{-1}\nabla_w^E l(z_t,w_t)$ could be
replaced with its intrinsic version  \eqref{manifold:grad:eq}
proposed in this paper. It can be conjectured this estimator
achieves Fisher efficiency i.e. reaches the \emph{intrinsic}
Cramer-Rao bound as defined in  \cite{smith-2005}. Such a result in
the theory of information geometry  goes  beyond the scope of this
paper and is left for future research.
\section*{Appendix B: Riemannian  geometry background}

Let $(\mathcal M,g)$ be a connected Riemannian manifold (see e.g. \cite{absil-book} for basic Riemannian geometry definitions).  It carries the structure of a metric space whose distance function is the arc length of a minimizing path between two points. The length $L$ of a curve $c(t)\in\mathcal M$ is defined by
$$
L=\int_a^b\sqrt{g(\dot c(t),\dot c(t))}dt=\int_a^b\norm{\dot c(t)}dt
$$If $y$ is sufficiently close to  $x\in\mathcal M$, there is a unique path of minimal length linking $x$ and $y$. It is called a geodesic. The exponential map is defined as follows: $\exp_x(v)$ is  the point $z\in\mathcal M$ situated on the  geodesic  with initial position-velocity $(x,v)$ at  distance $\norm v$ of $x$. We also define $\exp_x^{-1}(z)=v$. The cut locus of $x$ is roughly speaking the set where the geodesics starting at $x$ stop being paths of minimal length (for example $\pi$ on the circle for $x=0$). The least distance to the cut locus is the so-called injectivity radius $I$ at $x$. A geodesic ball is a ball with radius less than the injectivity radius at its center.

 For $f:\mathcal M\to \RR$ twice continuously differentiable, one can define the Riemannian gradient as the tangent vector at $x$ satisfying $\dotex|_{t=0} f(\exp_x(tv))=\langle v,\nabla f(x)\rangle_g$  and the  hessian as the operator  such that $\dotex|_{t=0}\langle \nabla f(\exp_x(tv)),\nabla f(\exp_x(tv))\rangle_g=2\langle \nabla f(x),(\nabla_x^2f) v\rangle_g.$ For instance, if $f(x)=\frac{1}{2}d^2(p,x)$ is half the squared distance to a point $p$ the Riemannian gradient is $\nabla_xf=\exp_x^{-1}(p)$, i.e.  it is a tangent vector at $x$ collinear to the geodesic linking  $x$ and $p$, with norm $d(p,x)$. Letting $c(t)=\exp_x(tv)$ we have
\begin{align*}
f(c(t))=f(x)&+t\langle v,\nabla f(x)\rangle_g\\&+\int_0^t(t-s)\langle \frac{d}{ds} c(s),(\nabla_{c(s)}^2f) \frac{d}{ds} c(s)\rangle_g ds.
\end{align*}and thus $f(\exp_x(tv))-f(x)\leq t\langle v,\nabla f(x)\rangle_g+\frac{t^2}{2}\norm{v}_g^2 k$, where $k$ is a bound on the hessian along the geodesic.

\bibliographystyle{plain}

\end{document}